\documentclass[10pt]{amsart}
\usepackage{amsmath,amscd}
\usepackage{amsbsy}
\usepackage{amssymb}
\usepackage{amscd,amsthm}
\usepackage[all,cmtip]{xy}

\newtheorem{thm}{Theorem}
\numberwithin{thm}{section}
\newtheorem{lem}[thm]{Lemma}
\newtheorem{prop}[thm]{Proposition}
\newtheorem{cor}[thm]{Corollary}
\newtheorem{exam}[thm]{Example}
\newtheorem{rema}[thm]{Remark}

\newtheorem{defi}[thm]{Definition}

\newtheorem*{thm2}{Theorem}
\newtheorem*{cor2}{Corollary}

\newtheorem*{prob2}{Problem}
\newtheorem*{que2}{Question}

\begin{document}
\begin{center}
\huge{Ulrich complexity and categorical representability dimension 
}\\[1cm]
\end{center}

\begin{center}

\large{Sa$\mathrm{\check{s}}$a Novakovi$\mathrm{\acute{c}}$}\\[0,4cm]
{\small April 2025}\\[0,3cm]
\end{center}

\noindent{\small \textbf{Abstract}. 
We investigate the Ulrich complexity of certain examples of Brauer--Severi varieties, twisted flags and involution varieties and establish lower and upper bounds. Furthermore, we relate Ulrich complexity to the categorical representability dimension of the respective varieties. We also state an idea why, in general, a relation between Ulrich complexity and categorical representability dimension may appear.}\\

\begin{center}
\tableofcontents
\end{center}

\section{Introduction}
Let $H$ a very ample line bundle on a variety $X$. In \cite{ESWV} the authors defined an \emph{Ulrich bundle} for $(X,H)$ to be a vector bundle $\mathcal{E}$ satisfying $h^q(X,\mathcal{E}(-iH))=0$ for each $q\in\mathbb{Z}$ and $1\leq i\leq \mathrm{dim}(X)$. The smallest $r>0$ such that there is a rank $r$ Ulrich bundle on $(X,H)$ is called the \emph{Ulrich complexity} of $X$ and will be denoted by $\mathrm{uc}((X,H))$ (see \cite{BESR}). Ulrich bundles are somehow the "nicest" arithmetically Cohen--Macaulay sheaves which are important to understand, since they give a measurement of the complexity of the variety. Ulrich bundles attracted a lot of attention in view of their multiple links to other topics such as Boij--S\"oderberg therory, Chow forms, matrix factorizations an so on. It was conjectured in \cite{ESWV} that any projective variety carries an Ulrich bundle. Moreover, it is asked for the smallest possible rank of such a bundle. This conjecture is a wide open problem, and we know a few result at the present. Varieties known to carry Ulrich sheaves include curves and Veronese varieties \cite{ESWV}, \cite{HV}, \cite{DF}, \cite{LR} complete intersections \cite{BHU1V}, generic linear determinantal varieties \cite{BHUV}, Segre varieties \cite{CMRPV}, rational normal scrolls \cite{MRV}, Grassmannians \cite{CMRV}, some flag varieties \cite{COS}, \cite{CMRV}, generic $\mathrm{K}$3 surfaces \cite{AFOV}, abelian surfaces \cite{BEV}, Enriques surfaces \cite{BNV}, ruled surfaces \cite{ACMRV} or twisted flags \cite{NOVA}, \cite{NV1} to mention only a few. The Ulrich complexity for some varieties is determined for instance in \cite{BESR} and \cite{DF}. Among others, in \cite{ESWV} it is proved that any curve $C\subset\mathbb{P}^n$ has a rank-$2$ Ulrich sheaf, provided the base field of the curve is infinite. Examples may be pointless conics defined by $x^2+y^2+z^2=0$. Recall that a scheme $X$ of finite type over a field $k$ is called \emph{Brauer--Severi variety} if $X\otimes_k\bar{k}\simeq \mathbb{P}^n$. Via Galois cohomology, isomorphism classes of Brauer--Severi varieties over $k$ are in one-to-one correspondence with isomorphism classes of central simple algebras over $k$. Moreover, Brauer--Severi varieties (respectively the corresponding central simple algebras) have important invariants called period, index and degree (see Section 2 for details). Pointless conics such as those defined by $x^2+y^2+z^2=0$ over a subfield of $\mathbb{R}$ are Brauer--Severi varieties, and since these admit rank-$2$ Ulrich bundles by \cite{ESWV}, it is natural to ask about Ulrich bundles over Brauer--Severi varieties of higher dimension and relate its rank to the arithmetic.
In the case of non-split Brauer--Severi curves there are always Ulrich bundles of rank two (see \cite{NV1}). It is an easy observation that non-split Brauer--Severi varieties cannot have Ulrich line bundles. So the minimal rank of an Ulrich bundle on a Brauer--Severi curve is one or two, depending on whether the curve admits a rational point or not. 

In \cite{NV1} it is shown that there exist always Ulrich bundles on any polarized Brauer--Severi variety. So the existence of Ulrich bundles for Brauer--Severi varieties has been solved completely. 
In view of the fact that not much is known about the minimal rank of Ulrich bundles for $(\mathbb{P}^n,\mathcal{O}_{\mathbb{P}^n}(d))$, it seems to be a challenging problem to determine the Ulrich complexity of Brauer--Severi varieties. An attempt was given by the author in \cite{NV1}. For instance, using a generalized Hartshorne--Serre correspondence, it was shown that certain Brauer--Severi varieties associated to a central simple division algebra of index $4$ and period $2$ admit a unique rank two Ulrich bundle. More precise, it is proved:
\begin{thm2}[\cite{NV1}, Theorem 5.15]
	Let $X$ be a Brauer--Severi variety of dimension $3$ where $\mathcal{O}_X(2)$ exists. Suppose there is a smooth geometrically connected genus one curve $C$ on $X$ such that the restriction map $H^0(X,\mathcal{O}_X(2))\rightarrow H^0(C,\mathcal{O}_C(2))$ is bijective. Then there is an  unique Ulrich bundle of rank two for $(X,\mathcal{O}_X(2))$.
\end{thm2}
After base change to the algebraic closure, it follows from \cite{ESWV}, Corollary 5.3 that the rank of an Ulrich bundle on $X$ from above must be divisible by two. Hence it cannot carry an Ulrich bundle of rank one. This implies: 
\begin{cor2}
Let $X$ be a Brauer--Severi variety of dimension $3$ where $\mathcal{O}_X(2)$ exists. Suppose there is a smooth geometrically connected genus one curve $C$ on $X$ such that the restriction map $H^0(X,\mathcal{O}_X(2))\rightarrow H^0(C,\mathcal{O}_C(2))$ is bijective. Then the Ulrich comlexity of $X$ is two.	
\end{cor2}
In \cite{NV1} it is also given an example of a Brauer--Severi variety satisfying the conditions of the latter theorem. 
Recall that any Fano threefold of index two carries a special rank two Ulrich bundle \cite{BV}. In particular, there is always a special rank two Ulrich bundle for $(\mathbb{P}^3,\mathcal{O}_{\mathbb{P}^3}(2))$. Since there is no rank one Ulrich bundle for $(\mathbb{P}^3,\mathcal{O}_{\mathbb{P}^3}(2))$, the minimal rank of an Ulrich bundle for $(\mathbb{P}^3,\mathcal{O}_{\mathbb{P}^3}(2))$ is two. The period of $\mathbb{P}^3$, considered as a trivial Brauer--Severi variety, is one. Notice that in \cite{LV} it is shown that there is a unique rank two Ulrich bundle for $(\mathbb{P}^2,\mathcal{O}_{\mathbb{P}^2}(2))$. These results together with the results obtained in \cite{NV1} led us to formulate some questions concerning Ulrich bundles on Brauer--Severi varieties (see \cite{NV1}). Let $X$ be an arbitrary (non-trivial) Brauer--Severi variety of period $\mathrm{per}(X)=p$.

\begin{itemize}
	\item[1)] Is there always an Ulrich bundle of rank $\mathrm{per}(X)$ for $(X,\mathcal{O}_X(p))$ ?
	\item[2)] Is the minimal rank of an Ulrich bundle for $(X,\mathcal{O}_X(p))$ exactly $\mathrm{per}(X)=p$ ?
	\item[3)] How does the minimal rank of an Ulrich bundle for $(X,\mathcal{O}_X(p\cdot d))$ depend on $d$ ?
	\item[4)]	Suppose the minimal rank of an Ulrich bundle does not equal $\mathrm{per}(X)$. Is there a formula involving the invariants period, index and degree that calculates the minimal rank ?  
\end{itemize}
The motivation for the present work is to try to give answers to the above questions. But we also want to relate the rank of an Ulrich bundle to another concept. In the introduction of \cite{NHABIL} the author observed a relation between the so called \emph{categorical representability dimension} and the Ulrich complexity of a Brauer--Severi variety. We now recall the definition and state the observation below. For the definition of a semiorthogonal decomposition or exceptional collections of a $k$-linear triangulated category see Section 4. For more details, we refer, for instance, to \cite{BBT} and references therein. 
In \cite{BBT} Bernardara and Bolognesi introduced the notion of categorical representability. A $k$-linear triangulated category $\mathcal{T}$ is said to be \emph{representable in dimension $m$} if there is a semiorthogonal decomposition $\mathcal{T}=\langle \mathcal{A}_1,...,\mathcal{A}_n\rangle$ and for each $i=1,...,n$ there exists a smooth projective connected variety $Y_i$ with $\mathrm{dim}(Y_i)\leq m$, such that $\mathcal{A}_i$ is equivalent to an admissible subcategory of $D^b(Y_i)$ (see \cite{AB1T}). We use the following notation
\begin{eqnarray*}
	\mathrm{rdim}(\mathcal{T}):=\mathrm{min}\{m\mid \mathcal{T}\  \textnormal{is representable in dimension m}\},
\end{eqnarray*}
whenever such a finite $m$ exists. Let $X$ be a smooth projective $k$-variety. One says $X$ is \emph{representable in dimension} $m$ if $D^b(X)$ is representable in dimension $m$. We will use the following notation:
\begin{eqnarray*}
	\mathrm{rdim}(X):=\mathrm{rdim}(D^b(X)).
\end{eqnarray*}

The notation of categorical representability dimension of the bounded derived category of coherent sheaves $D^b(X)$ was introduced after it had been asked whether the derived category can detect the existence of a $k$-rational point (see the introduction in \cite{AB1111}). An indeed, it can be shown that certain varieties $X$ admit $k$-rational points if and only if $\mathrm{rdim}(X)=0$ \cite{AB111}, \cite{AB1111}, \cite{BB111}, \cite{MBT111} and \cite{NV2}, \cite{NV3}, \cite{NV4}.

Let us now recall the observation from the introduction of \cite{NHABIL}: let $X$ be a Brauer--Severi variety of period $p$ and, for simplicity, let us denote the Ulrich complexity of $(X,\mathcal{O}_X(p))$ by $\mathrm{uc}(\phi_p(X))$, where $\phi_p$ denotes the embedding of $X$ into projective space via $\mathcal{O}_X(p)$ as it is formulated in Theorem 2.1 in Section 2. The results in \cite{LV}, \cite{NV2} and \cite{NV1} show that $X$ admits a $k$-rational point (which is equivalent to $X$ being $k$-rational) if and only if there exists an Ulrich line bundle for $(X,\mathcal{O}_X(p))$, i.e. if and only if $\mathrm{uc}(\phi_p(X))=1$. If we relate this to the result that $X$ admits a $k$-rational point if and only if $\mathrm{rdim}(X)=0$ (see \cite{NV2}), we obtain the following: let $X$ be a Brauer--Severi variety of period $p$. Then $\mathrm{uc}(\phi_p(X))=1$ if and only if $\mathrm{rdim}(X)=0$. This gives $\mathrm{rdim}(X)+1=\mathrm{uc}(\phi_1(X))$ in the split case. In the non-split case the situation gets more involved. If $C$ is a non-split Brauer--Severi curve, then $\mathrm{uc}(\phi_2(X))=2$ (see \cite{NV1}). On the other hand, $\mathrm{rdim}(C)=1$ (see \cite{NV2}). And if $X$ is a 3-dimensional non-split Brauer--Severi variety of period 2 over $\mathbb{R}$, then, according to \cite{NV1}, Theorem 5.15 and Remark 5.17, $\mathrm{uc}(\phi_2(X))=2$. But since $\mathrm{ind}(X)=2$, \cite{NV2}, Theorem 6.14 gives $\mathrm{rdim}(X)=1$. Furthermore, for the non-tivial Brauer--Severi variety from the above corollary, we have $\mathrm{uc}(\phi_p(X))=2$. Note that $\mathrm{rdim}(X)\leq \mathrm{ind}(X)-1$ for a Brauer--Severi variety $X$ with equality if for instance $\mathrm{ind}(X)\leq 3$ (see \cite{NV3}, Proposition 4.1 and Theorem 1.4). This implies that for $X$ as in the corollary over $\mathbb{R}$, we have $\mathrm{rdim}(X)+1=\mathrm{ind}(X)=2$. Hence $\mathrm{uc}(\phi_p(X))=\mathrm{rdim}(X)+1$. This yields
\begin{thm2}
	Let $X$ be one of the Brauer--Severi varieties from above. Then 
	\begin{eqnarray*}
		\mathrm{uc}(\phi_p(X))=\mathrm{rdim}(X)+1.
	\end{eqnarray*}
\end{thm2}  
\noindent
In all the examples form above, we had period equals index. So, with regard to the results summarized in \cite{NV1} and the results obtained in \cite{NV2}, \cite{NV3} and \cite{NV4}, we formulated the following question:
\begin{que2}
	Let $X$ be a Brauer--Severi variety with period $p$ equals index. Is it true that $\mathrm{uc}(\phi_p(X))=\mathrm{rdim}(X)+1$ ?
\end{que2}
This is a natural question since it is conjectured \cite{NV2} that $\mathrm{rdim}(X)=\mathrm{ind}(X)-1$ if period equals index. Moreover, the results in \cite{NV1} suggest $\mathrm{uc}(\phi_p(X))=\mathrm{period}(X)$ if period equals index. Note that for determining both $\mathrm{uc}(\phi_p(X))$ and $\mathrm{rdim}(X)$, we used a full weak exceptional collection consisting of vector bundles $\mathcal{V}_i$ satisfying $\mathrm{End}(\mathcal{V}_i)=A^{\otimes i}$, where $A$ is the central simple algebra corresponding to $X$. Essentially, this full weak exceptional collection gives the semiorthogonal decomposition obtained in \cite{BERV}. Again, we refer to Section 4 for a definition of full weak exceptional collection. In the next sections we discover that the above question has a negative answer and that the relation between $\mathrm{rdim}$ and $\mathrm{uc}$ seems to be more complicated (see Sections 6 and 7). As an ad hoc argument, consider a biquaternion algebra $D_1\otimes D_2$ of index $4$ and let $X$ be the corresponding Brauer--Severi variety. According to \cite{PB}, Appendix B, $\mathrm{rdim}(X)=2$. Therefore, $\mathrm{rdim}(X)+1=3$. Now from Example 5.5 we conclude that $\mathrm{uc}(\phi_2(X))$ cannot be three. In this case we have $\mathrm{uc}(\phi_p(X))\neq \mathrm{rdim}(X)+1$. Nonetheless, one can ask for a relation between these two numbers. 
In Section 6 we also provide a criterion for when $\mathrm{uc}=\mathrm{rdim}+1$ in the case of (generalized) Brauer--Severi varieties. But what happens for other than Brauer--Severi varieties? Recall that a smooth curve $C$ of genus $>0$ over a field $k$ without rational point admits an Ulrich bundle of rank two. It is well known that $\mathrm{rdim}(C)=1$ (since there are no non-trivial semiorthogonal decompositions). Therefore $\mathrm{rdim}(C)+1$ equals the minimal rank of an Ulrich bundle on $C$. So we wonder whether there is indeed a relation between the minimal rank of an Ulrich bundle and $\mathrm{rdim}$ in a more gerenal (arithmetic) setting. At this point, we want to mention another observation. In \cite{LV} it is constructed an Ulrich bundle of rank 3 for $(\mathbb{P}^2,\mathcal{O}(3))$. We believe that this rank 3 Ulrich bundle descents to a non-split Brauer--Severi surface $S$. And since $\mathrm{rdim}(S)=2$, we would have $\mathrm{uc}(\phi_3(X))=\mathrm{rdim}(S)+1$.

At least for Brauer--Severi varieties $X$, it seems to be an interesting problem to study the relation between the numbers $\mathrm{uc}(\phi_p(X))$ and $\mathrm{rdim}(X)$. Here one could also consider $X$ over fields for which period equals index and ask whether there is a relation to the period-index problem. Moreover, studying Brauer--Severi varieties could provide a testing ground for studying more general varieties. Note that similar questions can also be formulated for twisted forms of flags and involution varieties. For some classes of twisted flags, the existence of Ulrich bundles is proved in \cite{NOVA}. The case of involution varieties is treated in Section 5, where the existence of Ulrich bundles is proved. In Section 7 the special case of involution surfaces is studied in more detail and the Ulrich complexity and its relation to the categorical representability dimension is established. But one can also consider other surfaces such as del Pezzo surfaces, ruled surfaces, abelian surfaces and so on (see Section 7). We summarize some results from the literature and see that, again, $\mathrm{rdim}+1=\mathrm{uc}$ or $\mathrm{rdim}=\mathrm{uc}$ holds in these cases as well. Finally, in the Appendix it is outlined how the observations made for Brauer--Severi varieties, twisted flags and involution varieties might also hold in more general setting. In a more general setting one can also define the period and the index of a variety $X$. This leads us to formulate the following problem:
\begin{prob2}
	Let $X$ be a smooth projective variety over a field $k$, embedded into $\mathbb{P}$ by $\mathcal{L}$. Find the relation between $\mathrm{uc}((X,\mathcal{L}))$ and $\mathrm{rdim}(X)$, depending on period, index, polarization and the arithmetic of $X$. 
\end{prob2}
As mentioned above, categorical representability dimension is related to the existence of a rational point on $X$. It seems to us that Ulrich complexity also depends on the arithmetic of a variety, especially if the field $k$ is not algebraically closed. So from this point of view it is also reasonable to state the problem from above.\\

\noindent
The main results of the present paper are stated in Sections 5, 6 and 7.\\


\noindent{\small \textbf{Convetions}. Throughout this work $k$ is an arbitrary field. Moreover, $\bar{k}$ denotes an algebraic closure and $\bar{\mathcal{E}}$ the base change of a vector bundle $\mathcal{E}$ over $k$ to $\bar{k}$. The dimension of the cohomology group $H^i(X,\mathcal{F})$ as $k$ vector space is abbreviated by $h^i(\mathcal{F})$.} 



\section{Brauer--Severi varieties}
We recall the basics of Brauer--Severi varieties and central simple algebras and refer to \cite{GSV} and references therein. A \emph{Brauer--Severi variety} of dimension $n$ is a scheme $X$ of finite type over $k$ such that $X\otimes_k L\simeq \mathbb{P}^n$ for a finite field extension $k\subset L$. This definition of a Brauer--Severi variety is equivalent to the definition given in the introduction (see \cite{GSV}, Remark 5.12). A field extension $k\subset L$ such that $X\otimes_k L\simeq \mathbb{P}^n$ is called \emph{splitting field} of $X$. Clearly, the algebraic closure $\bar{k}$ is a splitting field for any Brauer--Severi variety. One can show that a Brauer--Severi variety always splits over a finite separable field extension of $k$ (see \cite{GSV}, Corollary 5.1.4). By embedding the finite separable splitting field into its Galois closure, a Brauer--Severi variety splits over a finite Galois extension of the base field $k$ (see \cite{GSV}, Corollary 5.1.5). It follows from descent theory that $X$ is projective, integral and smooth over $k$. If the Brauer--Severi variety $X$ is already isomorphic to $\mathbb{P}^n$ over $k$, it is called \emph{split}, otherwise it is called \emph{non-split}. There is a well-known one-to-one correspondence between Brauer--Severi varieties and central simple $k$-algebras. Recall that an associative $k$-algebra $A$ is called \emph{central simple} if it is an associative finite-dimensional $k$-algebra that has no two-sided ideals other than $0$ and $A$ and if its center equals $k$. If the algebra $A$ is a division algebra, it is called \emph{central division algebra}. For instance, a Brauer--Severi curve is associated to a quaternion algebra (see \cite{GSV}). Central simple $k$-algebras can be characterized by the following well-known fact (see \cite{GSV}, Theorem 2.2.1): $A$ is a central simple $k$-algebra if and only if there is a finite field extension $k\subset L$ such that $A\otimes_k L \simeq M_n(L)$ if and only if $A\otimes_k \bar{k}\simeq M_n(\bar{k})$.

The \emph{degree} of a central simple algebra $A$ is now defined to be $\mathrm{deg}(A):=\sqrt{\mathrm{dim}_k A}$. According to the Wedderburn Theorem, for any central simple $k$-algebra $A$ there is an integer $n>0$ and a division algebra $D$ such that $A\simeq M_n(D)$. The division algebra $D$ is also central and unique up to isomorphism. Now the degree of the unique central division algebra $D$ is called the \emph{index} of $A$ and is denoted by $\mathrm{ind}(A)$. It can be shown that the index is the smallest among the degrees of finite separable splitting fields of $A$ (see \cite{GSV}, Corollary 4.5.9). Two central simple $k$-algebras $A\simeq M_n(D)$ and $B\simeq M_m(D')$ are called \emph{Brauer equivalent} if $D\simeq D'$. Brauer equivalence is indeed an equivalence relation and one defines the \emph{Brauer group} $\mathrm{Br}(k)$ of a field $k$ as the group whose elements are equivalence classes of central simple $k$-algebras and group operation being the tensor product. It is an abelian group with inverse of the equivalence class of $A$ given by the equivalence class of $A^{op}$. The neutral element is the equivalence class of $k$. The order of a central simple $k$-algebra $A$ in $\mathrm{Br}(k)$ is called the \emph{period} of $A$ and is denoted by $\mathrm{per}(A)$. It can be shown that the period divides the index and that both, period and index, have the same prime factors (see \cite{GSV}, Proposition 4.5.13). Denoting by $\mathrm{BS}_n(k)$ the set of all isomorphism classes of Brauer--Severi varieties of dimension $n$ and by $\mathrm{CSA}_{n+1}(k)$ the set of all isomorphism classes of central simple $k$-algebras of degree $n+1$, there is a canonical identification $\mathrm{CSA}_{n+1}(k)=\mathrm{BS}_n(k)$ via non-commutative Galois cohomology (see \cite{GSV} for details). Hence any $n$-dimensional Brauer--Severi variety $X$ corresponds to a central simple $k$-algebra of degree $n+1$. In view of the one-to-one correspondence between Brauer--Severi varieties and central simple algebras one can also speak about the period of a Brauer--severi variety $X$. It is defined to be the period of the corresponding central simple $k$-algebra $A$.

Geometrically, the period of a Brauer--Severi variety $X$ can be interpreted as the smallest positive integer $p$ such that $\mathcal{O}_X(p)$ exists on $X$. In other words, if $X\otimes_k L\simeq \mathbb{P}^n$, then $p$ is the smallest positive integer such that $\mathcal{O}_{\mathbb{P}^n}(p)$ descends to a line bundle on $X$. 
Moreover, the Picard group of $X$ is isomorphic to $\mathbb{Z}$ and is generated by $\mathcal{O}_X(p)$. In the present note, we make use of the following fact concerning embeddings of a Brauer--Severi variety.
\begin{thm}[\cite{KANG}, Theorem 1 or \cite{LIV}, Corollary 3.6]
	Let $X$ be a Brauer--Severi variety of period $p$ over $k$. Then any line bundle $\mathcal{O}_X(pd)$ for $d\geq 1$ gives rise to an embedding
	\begin{eqnarray*}
		\phi_{pd}\colon X\longrightarrow \mathbb{P}^{N-1}, \;\; \text{where}\;\; N=\binom{\mathrm{dim}(X)+pd}{pd}.
	\end{eqnarray*} 
	After base change to a splitting field $L$ of $X$, this embedding becomes the $dp$-uple Veronese embedding $v_{pd}$ of $X\otimes_k L=\mathbb{P}_L^{\mathrm{dim}(X)}$ into $\mathbb{P}_L^{N-1}$.
\end{thm}

\section{Generalities on Ulrich bundles}
Let $H$ a very ample line bundle on a variety $X$. Recall that in \cite{ESWV} a vector bundle $\mathcal{E}$ on $X$ is defined to be an \emph{Ulrich bundle} for $(X,H)$ if it satisfies $h^q(X,\mathcal{E}(-iH))=0$ for each $q\in\mathbb{Z}$ and $1\leq i\leq \mathrm{dim}(X)$. Although there are further properties of Ulrich bundles, we will list only those needed in the present note. We refer the reader to \cite{BV} and \cite{ESWV} for details.
\begin{lem}[\cite{BV}, (3.6)]
	Let $\pi\colon X\rightarrow Y$ be a finite surjective morphism, $\mathcal{L}$ a very ample line bundle on $Y$ and $\mathcal{E}$ a vector bundle on $X$. Then $\mathcal{E}$ is an Ulrich bundle for $(X,\pi^*\mathcal{L})$ if and only if $\pi_*\mathcal{E}$ is an Ulrich bundle for $(Y,\mathcal{L})$. 
\end{lem}
In the special case where $X$ is a Brauer--Severi variety, we also have:
\begin{prop}
	Let $X$ be a Brauer--Severi variety of period $p$ with splitting field $L$ and $d\geq 1$. A vector bundle $\mathcal{E}$ is an Ulrich bundle for $(X,\mathcal{O}_X(pd))$ if and only if $\mathcal{E}\otimes_k L$ is an Ulrich bundle for $(X\otimes_k L,\mathcal{O}_{X\otimes_k L}(pd))$.
\end{prop}
\begin{proof}
	The Brauer--Severi variety $X$ is embedded into $\mathbb{P}^N$ via $\mathcal{O}_X(pd)$ with the morphism $\phi_{pd}$ given in Theorem 2.1. Note that $H^i(X,\mathcal{F})\otimes_k E\simeq H^i(X\otimes_k E,\mathcal{F}\otimes_k E)$ for any coherent sheaf $\mathcal{F}$ and any field extension $k\subset E$. The assertion then follows from the fact that $\mathcal{O}_X(pd)$ is very ample if and only if $\mathcal{O}_{X\otimes_k L}(pd)$ is (see \cite{LIV}, Lemma 3.2 (2)).
\end{proof}
\begin{prop}[see \cite{BV}, Corollary 3.2]
	Let $X\subset \mathbb{P}$ be a smooth variety of dimension $n$, carrying an Ulrich bundle $\mathcal{F}$ of rank $r$. For every $d\geq 1$, $(X,\mathcal{O}_X(d))$ carries an Ulrich bundle of rank $rn!$. 
\end{prop}
\section{Weak exceptional collections and semiorthogonal decompositions}
In this section we recall the definition of a full weak exceptional collection, a concept introduced in \cite{ORLOV} . 
Let $\mathcal{D}$ be a triangulated category and $\mathcal{C}$ a triangulated subcategory. The subcategory $\mathcal{C}$ is called \emph{thick} if it is closed under isomorphisms and direct summands. Note that there are different definitions of thick subcategory in the literature. For a subset $A$ of objects of $\mathcal{D}$ we denote by $\langle A\rangle$ the smallest full thick subcategory of $\mathcal{D}$ containing the elements of $A$. For a smooth projective variety $X$ over $k$, we denote by $D^b(X)$ the bounded derived category of coherent sheaves on $X$. Moreover, if $B$ is an associated $k$-algebra, we write $D^b(B)$ for the bounded derived category of finitely generated left $B$-modules.
\begin{defi}
	\textnormal{Let $A$ be a division algebra over $k$, not necessarily central. An object $\mathcal{E}^{\bullet}\in D^b(X)$ is called \emph{$A$-exceptional} if $\mathrm{End}(\mathcal{E}^{\bullet})=A$ and $\mathrm{Hom}(\mathcal{E}^{\bullet},\mathcal{E}^{\bullet}[r])=0$ for $r\neq 0$. By \emph{weak exceptional object}, we mean $A$-exceptional for some division algebra $A$ over $k$. 
		If $A=k$, the object $\mathcal{E}^{\bullet}$ is called \emph{exceptional}. } 
\end{defi}
\begin{defi}
	\textnormal{A totally ordered set $\{\mathcal{E}^{\bullet}_0,...,\mathcal{E}^{\bullet}_n\}$ of weak exceptional objects on $X$ is called an \emph{weak exceptional collection} if $\mathrm{Hom}(\mathcal{E}^{\bullet}_i,\mathcal{E}^{\bullet}_j[r])=0$ for all integers $r$ whenever $i>j$. A weak exceptional collection is \emph{full} if $\langle\{\mathcal{E}^{\bullet}_0,...,\mathcal{E}^{\bullet}_n\}\rangle=D^b(X)$ and \emph{strong} if $\mathrm{Hom}(\mathcal{E}^{\bullet}_i,\mathcal{E}^{\bullet}_j[r])=0$ whenever $r\neq 0$. If the set $\{\mathcal{E}^{\bullet}_0,...,\mathcal{E}^{\bullet}_n\}$ consists of exceptional objects it is called \emph{exceptional collection}.}
\end{defi}
The notion of a full exceptional collection is a special case of what is called a semiorthogonal decomposition of $D^b(X)$. Recall that a full triangulated subcategory $\mathcal{D}$ of $D^b(X)$ is called \emph{admissible} if the inclusion $\mathcal{D}\hookrightarrow D^b(X)$ has a left and right adjoint functor. 
\begin{defi}
	\textnormal{Let $X$ be a smooth projective variety over $k$. A sequence $\mathcal{D}_0,...,\mathcal{D}_n$ of full triangulated subcategories of $D^b(X)$ is called \emph{semiorthogonal} if all $\mathcal{D}_i\subset D^b(X)$ are admissible and $\mathcal{D}_j\subset \mathcal{D}_i^{\perp}=\{\mathcal{F}^{\bullet}\in D^b(X)\mid \mathrm{Hom}(\mathcal{G}^{\bullet},\mathcal{F}^{\bullet})=0$, $\forall$ $ \mathcal{G}^{\bullet}\in\mathcal{D}_i\}$ for $i>j$. Such a sequence defines a \emph{semiorthogonal decomposition} of $D^b(X)$ if the smallest full thick subcategory containing all $\mathcal{D}_i$ equals $D^b(X)$.}
\end{defi}
\noindent
For a semiorthogonal decomposition we write $D^b(X)=\langle \mathcal{D}_0,...,\mathcal{D}_n\rangle$.
\begin{rema}
	\textnormal{Let $\mathcal{E}^{\bullet}_0,...,\mathcal{E}^{\bullet}_n$ be a full weak exceptional collection on $X$. It is easy to verify that by setting $\mathcal{D}_i=\langle\mathcal{E}^{\bullet}_i\rangle$ one gets a semiorthogonal decomposition $D^b(X)=\langle \mathcal{D}_0,...,\mathcal{D}_n\rangle$. Notice that $\mathcal{D}_i\simeq D^b(\mathrm{End}(\mathcal{E}^{\bullet}_i))$, where the equivalence is obtained by sending the complex $\mathrm{End}(\mathcal{E}^{\bullet}_i)$ concentrated in degree $0$ to $\mathcal{E}^{\bullet}_i$.}
\end{rema}

\section{Ulrich complexity of Brauer--Severi varieties and twisted flags}
The following section contains some observations that enable us to tackle questions 1), 2) and the question whether $\mathrm{uc}(\phi_p(X))=\mathrm{rdim}(X)+1$ if period equals index. Throughout this section, we want to relate the Ulrich complexity of $X$ to that of $X\otimes_k E$ for a suitable field extension $k\subset E$.  
\begin{thm}
	Let $X$ be a Brauer--Severi variety of dimension $n$ over a field $k$. Denote by $p$ the period of $X$ and fix an arbitrary integer $d\geq 1$. Let $E$ be a minimal separable splitting field of $X$. Then
	\begin{eqnarray*}
		\mathrm{uc}(v_{pd}(\mathbb{P}^n_{\bar{k}}))\leq \mathrm{uc}(v_{pd}(\mathbb{P}_E^n))\leq \mathrm{uc}(\phi_{pd}(X)))\leq \mathrm{ind}(X)\cdot \mathrm{uc}(v_{pd}(\mathbb{P}_E^n)).
		\end{eqnarray*}
	If every prime $\leq n$ divides $pd$, one has 
	\begin{eqnarray*}
		n!\leq \mathrm{uc}(\phi_{pd}(X))\leq \mathrm{ind}(X)\cdot n!
	\end{eqnarray*}
\end{thm}
\begin{proof}
 Notice that $\phi_{pd}(X)_E\simeq v_{pd}(\mathbb{P}^n_E)$, according to Theorem 2.1. It is well known that there is a Ulrich bundle  on $v_{pd}(\mathbb{P}^n_E)$ of rank $n!$ (see \cite{BV}). Now take an Ulrich bundle $\mathcal{E}$ of minimal rank. Denote the rank of $\mathcal{E}$ by $\mathrm{uc}(v_{pd}(\mathbb{P}^n))$. Since the map $\pi:X\otimes_k E\rightarrow X$ is finite, Lemma 3.1 provides us with an Ulrich bundle $\pi_*\mathcal{E}$ for $(X,\mathcal{O}_X(pd))$. The rank of $\pi_*\mathcal{E}$ is given as $\mathrm{rk}(\mathcal{E})\cdot \mathrm{ind}(X)$. 
	Hence $\mathrm{rk}(\pi_*\mathcal{E})=\mathrm{ind}(X)\cdot \mathrm{uc}(v_{pd}(\mathbb{P}_E^n))$. This shows
	\begin{eqnarray*}
		\mathrm{uc}(\phi_{pd}(X))\leq \mathrm{ind}(X)\cdot \mathrm{uc}(v_{pd}(\mathbb{P}^n_E)).
		\end{eqnarray*}
	
	On the other hand, there can not be an Ulrich bundle $\mathcal{F}$ on $\phi_{pd}(X)$ of rank strictly smaller than $\mathrm{uc}(v_{pd}(\mathbb{P}_E^n))$, since $\mathcal{F}\otimes_k E$ would give an Ulrich bundle on $v_{pd}(\mathbb{P}^n_E)$ of rank strictly smaller than $	\mathrm{uc}(v_{pd}(\mathbb{P}_E^n))$ (according to Proposition 3.2), which is impossible. It is an easy observation that $\mathrm{uc}(v_{pd}(\mathbb{P}^n_K))\leq \mathrm{uc}(v_{pd}(\mathbb{P}_k^n))$ for any finite field extension $k\subset K$. Hence $\mathrm{uc}(v_{pd}(\mathbb{P}^n_{\bar{k}}))\leq \mathrm{uc}(v_{pd}(\mathbb{P}_E^n))$. 
	This proves the first statement. 
	The second statement follows from \cite{BES}, Theorem 5.1. For the convenience of the reader, let us recall the argument. Assuming that every prime $q \leq n$ divides $pd$, we get 
	\begin{eqnarray*}
		\prod_{j=1}^n(pdj+1)\equiv 1 \ \mathrm{mod} \ q.
		\end{eqnarray*}
	Now, if $\mathcal{F}$ is an Ulrich bundle on $v_{pd}(\mathbb{P}_{\bar{k}}^n)$, the Euler chracteristic of $v_{pd}^{*}\mathcal{F}$ is given by
	\begin{eqnarray*}
		\chi (v_{pd}^*\mathcal{F}(t))=\frac{\mathrm{rk}(\mathcal{F})}{n!}\prod_{j=1}^n(pdj+t).
		\end{eqnarray*}
	We conclude that the $q$ part of $n!$ divides $\mathrm{rk}(\mathcal{F})$. This implies $n!\leq \mathrm{uc}(v_{pd}(\mathbb{P}^n_{\bar{k}}))$. Since $\mathrm{uc}(v_{pd}(\mathbb{P}_{\bar{k}}^n))= n!$, the second statement follows. 
\end{proof}
\begin{rema}
	\textnormal{Let $X$ be as in Theorem 5.1. Since $\mathrm{uc}(v_{pd}(\mathbb{P}^n_k))\leq n!$ for any field $k$, we get 
	\begin{eqnarray*}
		\mathrm{uc}(v_{pd}(\mathbb{P}_{\bar{k}}^n))\leq \mathrm{uc}(\phi_{pd}(X)))\leq \mathrm{ind}(X)\cdot  n!
	\end{eqnarray*}
	If $k=\mathbb{R}$, it follows  
	\begin{eqnarray*}
		\mathrm{uc}(v_{pd}(\mathbb{P}_{\mathbb{C}}^n))\leq \mathrm{uc}(\phi_{pd}(X)))\leq 2\cdot \mathrm{uc}(v_{pd}(\mathbb{P}_{\mathbb{C}}^n))\leq 2n!
	\end{eqnarray*} 
since the minimal seperable splitting field is $\mathbb{C}$.}
\end{rema}
\begin{exam}
	\textnormal{Let $X$ be a non-split Brauer--Severi curve over $k$. The period of $X$ is two and $\mathrm{uc}(\phi_{2d}(X)))=2$ (see \cite{NV1}). This also follows from Remark 5.2. Indeed, we have $\mathrm{uc}(v_{2d}(\mathbb{P}^1_{\bar{k}}))\leq \mathrm{uc}(\phi_{2d}(X))\leq 2$. Since there cannot be a rank one Ulrich bundle on a non-split Brauer--Severi curve, it follows that $\mathrm{uc}(\phi_{2d}(X)))=2$.}
\end{exam}
\begin{exam}
	\textnormal{Let $X$ be non-split Brauer--Severi variety surface over $k$. The degree of the corresponding central simple algebra is three and therefore $p=3$. Since period equals index, and since $\mathrm{uc}(v_{3d}(\mathbb{P}_{\bar{k}}^2))=2$ (see \cite{CMR}), we have $2\leq \mathrm{uc}(\phi_{3d}(X))\leq 6$, according to Theorem 5.1. If $3d$ is even, we have $\mathrm{uc}(\phi_{3d}(X))\in\{2,4,6\}$. If $3d$ is odd, we have  $\mathrm{uc}(\phi_{3d}(X))\in\{2,3,4,5,6 \}$. This follows from the results in \cite{CMR} (see also Introduction of \cite{DF}). 
	Note that $\mathrm{rdim}(X)=2$, according to \cite{NV3}, Theorem 1.4. This shows that $\mathrm{rdim}(X)+1=3\neq  \mathrm{uc}(\phi_{3d}(X))$ if $3d$ is even. } 
\end{exam}
\begin{exam}
	\textnormal{Let $X$ be a non-split Brauer--Severi variety of dimension three. Then $p=2,4$. So we consider $\mathrm{uc}(\phi_{2d}(X))$ and $\mathrm{uc}(\phi_{4d}(X))$. Let us denote by $\bar{g}\in \{0,2,4\}$ the remainder of the division of $pd$ by $6$. We first consider the case where period equals two. If $\bar{g}=2,4$, \cite{DF}, Theorem 1 and Theorem 5.1 from above tell us that $\mathrm{uc}(\phi_{pd}(X))\in \{2,4\}$. If $\bar{g}=0$, we obtain $\mathrm{uc}(\phi_{pd}(X))\in \{6,12\}$. Now if the period equals four and $\bar{g}=2,4$, we find $\mathrm{uc}(\phi_{pd}(X))\in \{2,4,6,8\}$. If $\bar{g}=0$, we have $\mathrm{uc}(\phi_{pd}(X))\in \{6,12,24\}$. Summarizing, we finally obtain 
	\begin{eqnarray*}
		\mathrm{uc}(\phi_{pd}(X))\in \begin{cases}\{2,4\} & \text{if} \;\;\;\;\; p=2, \bar{g}=2,4,\\
			\{6,12\} & \text{if} \;\;\;\;\; p=2, \bar{g}=0,\\
			\{2,4,6,8\}& \text{if} \;\;\;\;\;\; p=4, \bar{g}=2,4,\\
			\{6,12,24\}& \text{if} \;\;\;\;\;\; p=4, \bar{g}=0.
		\end{cases}
	\end{eqnarray*}
Now, let us consider $\mathrm{rdim}(X)$. Let $A$ be the central simple algebra corresponding to $X$. Recall that $\mathrm{deg}(A)=$ and since $p=2,4$, we have $\mathrm{ind}(A)\in\{2,4\}$. Now, from \cite{NV3}, Theorem 1.4, we conclude $\mathrm{rdim}(X)+1\in\{2,4\}$. So in the case $p=\mathrm{ind}(A)=4$ and $pd=12$ we have $\mathrm{rdim}(X)+1\neq \mathrm{uc}(\phi_{pd}(X))$.  }
\end{exam}
\begin{exam}
	\textnormal{For a non-split Brauer--Severi variety in dimension $\geq 4$ not much can be said. If $\mathrm{dim}(X)=5$ and $p=2$, one has $8\leq \mathrm{uc}(\phi_{2}(X))$ (see \cite{DF}, Remark 3). Now consider central simple algebras over $\mathbb{R}$. Let $D$ be a quaternion algebra over $\mathbb{R}$ and $A=\mathrm{Mat}_3(D)$. Then $\mathrm{deg}(A)=6$ and hence $A$ corresponds to a Brauer--Severi variety of dimension $5$. Its period and index is two. According to \cite{NV3}, Proposition 4.1 it follows $\mathrm{rdim}(X)=1$. And since $8\leq \mathrm{uc}(\phi_{2}(X))$, we find $\mathrm{uc}(\phi_{2}(X))\neq \mathrm{rdim}(X)+1$. There is also an alternative argument using the main result in \cite{LR}. More precisely, \cite{LR}, Theorem 1 implies that $4\leq \mathrm{uc}(\phi_{2}(X))$. Again, since $\mathrm{rdim}(X)=1$, we obtain $\mathrm{uc}(\phi_{2}(X))\neq \mathrm{rdim}(X)+1$. This gives a negative answer to the Question mentioned in the introduction. }
\end{exam}
\begin{cor}
	Let $p$ be prime and $X$ a non-split Brauer--Severi variety of dimension $p-1$. Set $d=(p-1)!$. Then 
	\begin{eqnarray*}
		\mathrm{rdim}(X)<\mathrm{uc}(\phi_{p!}(X)).
		\end{eqnarray*}
\end{cor}
\begin{proof}
	This follows from Theorem 5.1 and \cite{NV3}, Proposition 4.1 and the fact that period equals index in this situation.  Indeed, we have
	\begin{eqnarray*}
		\mathrm{rdim}(X)\leq p-1<(p-1)!\leq \mathrm{uc}(\phi_{p!}(X)).
	\end{eqnarray*}	
\end{proof}
\begin{prop}
Let $X$ be a non-split Brauer--Severi variety of period $p$. If $\mathrm{dim}(X)\geq 4$, then 
\begin{eqnarray*}
	4\leq \mathrm{uc}(\phi_{pd}(X))).
	\end{eqnarray*}	
\end{prop}
\begin{proof}
This follows from Theorem 5.1 and \cite{LR}, Theorem 1.
\end{proof}
This proposition provides us with examples of Brauer--Severi varieties $X$ of period $p=2,3$ for which there are no rank $p$ Ulrich bundles on $(X,\mathcal{O}_X(p))$. This fact immediately gives negative answers to questions 1) and 2). Take for instance the quaternion algebra $D$ over $\mathbb{R}$ from Example 5.6. The Brauer--Severi variety $X$ corresponding to $A=\mathrm{Mat}_3(D)$ is of dimension 5 and period 2. Moreover, period equals index in this case. Proposition 5.8 tells us that there is no Ulrich bundle of rank 2.  

Note that $\mathrm{rdim}(X)\leq \mathrm{ind}(X)-1$ for a Brauer--Severi variety $X$ with equality if for instance $\mathrm{ind}(X)\leq 3$ (see \cite{NV3}, Proposition 4.1 and Theorem 1.4). It is therefore sensible to formulate:
\begin{cor}
		Let $X$ be a Brauer--Severi variety of dimension $n$ over a field $k$. Denote by $p$ the period of $X$ and fix an arbitrary integer $d\geq 1$. Let $E$ be a minimal separable splitting field of $X$ and assume $\mathrm{rdim}(X)+1=\mathrm{ind}(X)$. Then 
		\begin{eqnarray*}
		\mathrm{uc}(v_{pd}(\mathbb{P}_{\bar{k}}^n))\leq \mathrm{uc}(\phi_{pd}(X)))\leq (\mathrm{rdim}(X)+1)\cdot \mathrm{uc}(v_{pd}(\mathbb{P}^n)).
	\end{eqnarray*}
	If every prime $\leq n$ divides $pd$, one has 
	\begin{eqnarray*}
		n!\leq \mathrm{uc}(\phi_{pd}(X))\leq (\mathrm{rdim}(X)+1)\cdot n!
	\end{eqnarray*}
\end{cor}

Notice that Theorem 5.1 also holds true for certain twisted flag varieties. We recall briefly the definition of twisted flags and refer to \cite{MPW111} for details. Let $G$ be a semisimple algebraic group over a field $k$ and $G_s=G\otimes_k k^{sep}$. For a parabolic subgroup $P$ of $G_s$, one has a homogeneous variety $G_s/P$. A \emph{twisted flag} is variety $X$ such that $X\otimes_k k^{sep}$ is $G_s$-isomorphic to $G_s/P$ for some $G$ and some parabolic $P$ in $G_s$. Any twisted flag is smooth, absolutely irreducible and reduced. An algebraic group $G'$ is called twisted form of $G$ iff $G'_s\simeq G_s$ iff $G'={_\gamma} G$ for some $\gamma\in Z^1(k,\mathrm{Aut}(G_s))$. The group $G'$ is called an \emph{inner form} of $G$ if there is a $\delta\in Z^1(k,\bar{G}(k^{sep}))$ with $G'={_\delta}G$. Here $\bar{G}=G/Z(G)$. For an arbitrary semisimple $G$ over $k$, there is a unique (up to isomorphism) split semisimple group $G^d$ such that $G_s\simeq G_s^d$. If $G$ is an inner form of $G^d$, then $G$ is said to be of \emph{inner type}. For instance, let $A$ be a central simple algebra over $k$ of degree $n$ and $G=\mathrm{PGL}_1(A)$, then $G_s\simeq \mathrm{PGL}_n$ over $k^{sep}$. Hence $G$ is an inner form of $\mathrm{PGL}_n$. Since $\mathrm{PGL}_n$ is split, $G=\mathrm{PGL}_1(A)$ is of inner type. The inner twisted forms arising from $G=\mathrm{PGL}_1(A)$ can be described very explicitly (see \cite{MPW111}, Section 5). One of these inner twisted forms is the \emph{generalized Brauer--Severi variety}. 
So let $m\leq n$. The generalized Brauer--Severi variety $\mathrm{BS}(m,A)$ is defined to be the subset of $\mathrm{Grass}_k(mn,A)$ consisting of those subspaces of $A$ which are right ideals of dimension $m\cdot n$ (see \cite{BLA}). 
Recall that $\mathrm{Grass}_k(mn,A)$ is given the structure of a projective variety via the Pl\"ucker embedding (see \cite{BLA})
$$
	\mathrm{Grass}_k(mn,A)\longrightarrow \mathbb{P}(\wedge^{mn}(A)).
$$	
This gives an embedding $\mathrm{BS}(m,A)\rightarrow \mathbb{P}(\wedge^{mn}(A))$ and a very ample line bundle $\mathcal{M}$.
Note that for any $\mathrm{BS}(m,A)$ there exists a finite Galois field extension $E$ such that $\mathrm{BS}(m,A)\otimes_k E\simeq \mathrm{Grass}_E(mn,n^2)\simeq \mathrm{Grass}_E(m,n)$. The Picard group $\mathrm{Pic}(\mathrm{Grass}_E(m,n))$ is isomorphic to $\mathbb{Z}$ and has ample generator $\mathcal{O}(1)\simeq \mathrm{det}(\mathcal{Q})$ with $\mathcal{Q}$ being the universal quotient bundle on $\mathrm{Grass}_E(m,n)$. Recall that $\mathrm{Pic}(\mathrm{BS}(m,A))\simeq\mathbb{Z}$ and that there is a positive generator $\mathcal{L}$ such that $\mathcal{L}\otimes_k E\simeq \mathcal{O}(r)$ for a suitable $r>0$. Since $\mathrm{Pic}(\mathrm{BS}(m,A))$ is cyclic, we have $\mathcal{L}^{\otimes s}\simeq \mathcal{M}$ for a suitable $s>0$. Therefore, $\mathcal{L}$ is ample. From the definition of $\mathrm{BS}(m,A)$ it is clear that $\mathcal{L}$ is also very ample.\\
In general, the inner twisted flags arising from $G=\mathrm{PGL_1}(A)$, where $A$ is a central simple algebra of degree $n$, are varieties denoted by $\mathrm{BS}(n_1,...,n_l,A)$, with $n_1<\cdots<n_l<n$, satisfying $\mathrm{BS}(n_1,...,n_l,A)\otimes_k k^s\simeq \mathrm{Flag}_{k^s}(n_1,...,n_l;n)$. These partial twisted flags parametrize sequences $I_1\subseteq\cdots \subset I_l\subseteq A$ of right ideals with $\mathrm{dim}(I_j)=n\cdot n_j$, for $j=1,...,l$ (see \cite{MPW111}, Section 5). If $E$ is a splitting field of $A$, i.e $A\otimes_k E\simeq M_n(E)$, one has $\mathrm{BS}(n_1,...,n_l,A)\otimes_k E\simeq \mathrm{Flag}_E(n_1,...,n_l;n)$. For details, we refer to \cite{MPW111}. Recall that a flag $\mathbb{F}_L:=\mathrm{Flag}_L(n_1,...,n_l;n)$ over a field $L$ has $l$ projections $p_i\colon \mathrm{Flag}_L(n_1,...,n_l;n)\longrightarrow \mathrm{Grass}_L(l_i,n)$. The Picard group of $\mathbb{F}_L$ is generated by $\mathcal{L}_i=p_i^*\mathcal{O}_{\mathrm{Grass}_{L}(l_i,n)}(1)$. A line bundle $\mathcal{R}$ on $\mathbb{F}_L$ is ample if and only if $\mathcal{R}=\mathcal{L}_1^{\otimes a_1}\otimes\cdots\otimes\mathcal{L}_l^{\otimes a_l}$ with $a_i>0$. We set $\mathcal{O}_{\mathbb{F}_L}(1)=\mathcal{L}_1\otimes\cdots\otimes\mathcal{L}_l$.  
 \begin{prop}[modification of \cite{NOVA}, Proposition 7.]
Let $X$ be a smooth projective geometrically integral variety over a field $k$ and $k\subset E$ be a finite separable field extension. Assume $X\otimes_k E$ is embedded into projective space via $\mathcal{O}_{X\otimes_k E}(1)$, i.e $\mathcal{O}_{X\otimes_k E}(1)=i^*\mathcal{O}_{\mathbb{P}^N}(1)$ for an embedding $i\colon X\otimes_k E\rightarrow \mathbb{P}^N$. Then there is a very ample line bundle $\mathcal{L}$ on $X$, satisfying $\mathcal{L}\otimes_k E\simeq \mathcal{O}_{X\otimes_k E}(r)$ for a suitable $r>0$.  
\end{prop}
 \begin{proof}
	We restate the proof of \cite{NOVA}, Proposition 7 and make minor modifications. Let $G$ be the absolute Galois group. It is well know that there is an exact sequence arising from the Leray spectral sequence 
	\begin{eqnarray*}
		0\longrightarrow \mathrm{Pic}(X)\longrightarrow \mathrm{Pic}(X\otimes_k k^{sep})^G\stackrel{\delta}\longrightarrow \mathrm{Br}(k)\longrightarrow \mathrm{Br}(k(X)).
	\end{eqnarray*}
	Note that every element of $\mathrm{Br}(k)$ has finite order. Now let $\mathcal{O}_{X_E}(m):=\mathcal{O}_{X\otimes_k E}(m)$ with $m>0$ be an element of $\mathrm{Pic}(X\otimes_k E)$ and consider $\mathcal{O}_{X_E}(m)\otimes_E k^{sep}\simeq \mathcal{O}_{X\otimes_k E\otimes_E k^{sep}}(m)$. Now let this line bundle be in the cokernel of the map $\mathrm{Pic}(X)\rightarrow \mathrm{Pic}(X\otimes_k k^{sep})^G$. If this element is trivial, we are done. Indeed, in this case there is an (ample) line bundle $\mathcal{L}$ on $X$ such that $\mathcal{L}\otimes_k k^{sep}\simeq \mathcal{O}_{X\otimes_k E\otimes_E k^{sep}}(m)$. Now, it is easy to see that $\mathcal{L}\otimes _k E\simeq \mathcal{O}_{X_E}(m)$. 
	Assume $\mathcal{O}_{X\otimes_k E\otimes_E k^{sep}}(m)$ is non-trivial. Then $\delta(\mathcal{O}_{X\otimes_k E\otimes_E k^{sep}}(m))$ is a non-trivial Brauer-equivalence class $[B]\in\mathrm{Br}(k)$. If $d>0$ is the order of $[B]$ in $\mathrm{Br}(k)$, we obtain that $\delta(\mathcal{O}_{X\otimes_k E\otimes_E k^{sep}}(md))=[k]$. This implies that there exists a line bundle $\mathcal{L}$ on $X$ such that $\mathcal{L}\otimes_k k^{sep}\simeq\mathcal{O}_{X\otimes_k E\otimes_E k^{sep}}(md)$. Again, it is easy to see that $\mathcal{L}\otimes_k E\simeq \mathcal{O}_{X_E}(md)$. Since $\mathcal{O}_{X\otimes_k E\otimes_E  k^{sep}}(md)$ is very ample, we conclude that $\mathcal{L}$ must be very ample. Therefore, there is a very ample line bundle $\mathcal{L}$ on $X$, satisfying $\mathcal{L}\otimes_k E\simeq \mathcal{O}_{X\otimes_k E}(r)$ for a suitable $r>0$.
\end{proof}
\begin{thm}[Type $A_n$]
	Let $A$ be a degree $n$ central simple algebra over a field $k$ of characteristic zero and $E$ a minimal separable splitting field of $A$. Let $X$ be one of the following twisted flag varieties:
	\begin{eqnarray*}
		\mathrm{BS}(m,A),\;\; \mathrm{BS}(1,n-1,A),\;\; \mathrm{BS}(1,n-2,A),\;\; \mathrm{BS}(2,n-2,A),\;\;\\ 
		\mathrm{BS}(m,m+1,A),\;\;\mathrm{BS}(m,m+2,A).
	\end{eqnarray*}
Then there is a very ample line bundles $\mathcal{L}$ on $X$, satisfying $\mathcal{L}\otimes_k E\simeq \mathcal{O}_{\mathbb{F}_E}(r)$ for suitable $r>0$, such that for any $d>0$ 
\begin{eqnarray*}
	\mathrm{uc}((\mathbb{F}_{\bar{k}},\mathcal{O}_{\mathbb{F}_{\bar{k}}}(rd)))\leq \mathrm{uc}((\mathbb{F}_E,\mathcal{O}_{\mathbb{F}_E}(rd)))\leq \mathrm{uc}((X,\mathcal{L}^{\otimes d}))\leq \mathrm{ind}(A)\cdot  \mathrm{uc}((\mathbb{F}_E,\mathcal{O}_{\mathbb{F}_E}(rd))).
\end{eqnarray*}
\end{thm}
\begin{proof}
See \cite{NOVA} for the exsistence of Ulrich bundles on $X$ and references therein for the existence of Ulrich bundles on the flags $(\mathbb{F}_E,\mathcal{O}_{\mathbb{F}_E}(rd))$. The rest of the proof follows exactly the lines of the proof of Theorem 5.1 and uses Proposition 5.10.	
\end{proof}
Notice that $\mathrm{rdim}(X)+1\leq \mathrm{ind}(A)$ for $X=\mathrm{BS}(m,A)$ with equality if for instance $\mathrm{ind}(A)\leq 3$ (see \cite{NV2}, Proposition 6.14 and Theorem 6.15). 
\begin{cor}
Let $A$ be a central simple algebra over a field $k$ of characteristic zero and $X=\mathrm{BS}(m,A)$. Furthermore, let $E$ be a minimal splitting field of $A$ and assume $\mathrm{rdim}(X)+1=\mathrm{ind}(A)$. Then there is a very ample line bundles $\mathcal{L}$ on $X$, satisfying $\mathcal{L}\otimes_k E\simeq \mathcal{O}_{\mathbb{F}_E}(r)$ for suitable $r>0$, such that for any $d>0$ 
\begin{eqnarray*}
	\mathrm{uc}((\mathbb{F}_E,\mathcal{O}_{\mathbb{F}_E}(rd)))\leq \mathrm{uc}((X,\mathcal{L}^{\otimes d}))\leq (\mathrm{rdim}(X)+1)\cdot  \mathrm{uc}((\mathbb{F}_E,\mathcal{O}_{\mathbb{F}_E}(rd))).
\end{eqnarray*}	
\end{cor}
Let $Y$ be one of the following flags: symplectic Grassmannians $\mathrm{IGrass}_{k^{sep}}(2,2n)$ for $n\geq 2$, odd and even-dimensional quadrics, the orthogonal Grassmannians $\mathrm{OGras}_{k^{sep}}(2,m)$ for $m\geq 4$, $\mathrm{OGrass}_{k^{sep}}(3,4q+6)$ for $q\geq 0$ and $\mathrm{OGrass}_{k^{sep}}(4,8)$. Then $\mathrm{Pic}(Y)\simeq \mathbb{Z}$ is generated by a very ample line bundle (see \cite{FO}, Section 2). We will denote this very ample generator by $\mathcal{O}_Y(1)$. For details on inner twisted forms of $Y$, we again refer to \cite{MPW111}, Section 5. Now let $X$ be an inner twisted form of $Y$. Notice that for any $X$ there is a finite (separable) field extension $E$ of $k$ (called splitting field) such that $X\otimes_k E$ is isomorphic to the corresponding flag. The minimal degree of such an splitting field is also denoted by $\mathrm{ind}(X)$.

\begin{thm}[Type $B_n,C_n$ and $D_n$]
	Let $X$ be an inner twisted form of $Y$ and $E$ a minimal splitting field. Then there is a very ample line bundles $\mathcal{L}$ on $X$, satisfying $\mathcal{L}\otimes_k E\simeq \mathcal{O}_{Y}(r)$ for suitable $r>0$, such that for any $d>0$ 
	\begin{eqnarray*}
		\mathrm{uc}((Y,\mathcal{O}_{Y}(rd)))\leq \mathrm{uc}((X,\mathcal{L}^{\otimes d}))\leq \mathrm{ind}(X)\cdot  \mathrm{uc}((Y,\mathcal{O}_{Y}(rd))).
	\end{eqnarray*}
\end{thm}
\begin{proof}
	Analogous to the proof of Theorem 5.11.
\end{proof}

\begin{rema}
\textnormal{The lower and upper bound from Theorems 5.11 and 5.13 can be made quite explicit, using the results from \cite{CMRV} and \cite{COS}. }	
	\end{rema}
In the case of certain twisted quadrics we have also an analogue of Theorem 5.1, formulated in Theorem 5.14 below. Let us first recall some basic facts concerning twisted quadrics. {Let $G=\mathrm{PSO}_n$ with $n$ even. In this case $\widetilde{G}= \mathrm{Spin}_n$. Consider the action of $G$ on $\mathbb{P}^{n-1}$ given by projective linear transformations. We write $P\subset G$ for the stabilizer of the point $[1:0:\cdots :0]$. The projective homogeneous variety $G/P$ is a smooth quadric hypersurface $Q\subset \mathbb{P}^{n-1}$. Given a 1-cocycle $\gamma\colon \mathrm{Gal}(k^s|k)\rightarrow \mathrm{PSO}_n(k^s)$, it is well known that $\gamma$ determines a central simple $k$-algebra $A$ with an involution $\sigma$ of orthogonal type. The associated twisted homogeneous space ${_\gamma}(G/P)$ is a twisted form of the quadric $G/P$.
	
To a central simple algebra $A$ of degree $n$ with involution $\sigma$ of the first kind over a field $k$ of $\mathrm{char}(k)\neq 2$ one can associate the \emph{involution variety} $\mathrm{IV}(A,\sigma)$. This variety can be described as the variety of $n$-dimensional right ideals $I$ of $A$ such that $\sigma(I)\cdot I=0$. If $A$ is split so $(A,\sigma)\simeq (M_n(k), q^*)$, where $q^*$ is the adjoint involution defined by a quadratic form $q$ one has $\mathrm{IV}(A,\sigma)\simeq V(q)\subset \mathbb{P}^{n-1}_k$. Here $V(q)$ is the quadric determined by $q$. By construction, for a suitable $N$, such an involution variety  $\mathrm{IV}(A,\sigma)$ becomes a quadric in $\mathbb{P}^{N-1}_L$ after base change to some splitting field $L$ of $A$. In this way the involution variety is a twisted form of a smooth quadric as described before. The case of involution surfaces is treated in section 7. But let us first settle the existence problem for Ulrich bundles on involution varieties. 
\begin{prop}
	Let $X$ be an involution variety as above over a field $k$. Then $X$ carries an Ulrich bundle for any polarization.
\end{prop}	
\begin{proof}
	For a smooth quadric $Q\subset \mathbb{P}^{n-1}$ embedded by $\mathcal{O}_Q(1)$, it is well known that there are Ulrich bundles for $(Q,\mathcal{O}_Q(1)))$. Furthermore, it is known that $\mathrm{uc}((Q,\mathcal{O}_Q(1))=2^{\lfloor (n-3)/2\rfloor}$ (see \cite{BV}). With Proposition 3.3, we obtain an Ulrich bundle on $(Q,\mathcal{O}_Q(d))$. Moreover, one has $\mathrm{uc}((Q,\mathcal{O}_Q(d))\leq (n-2)!2^{\lfloor (n-3)/2\rfloor}$. Now let $X$ be embedded into a projective space by $\mathcal{L}^{\otimes d}$, where $\mathcal{L}$ is an ample line bundle. After base change to some finite field extension $k\subset E$, the involution variety $X$ becomes a quadric. Hence there there is an Ulrich bundle $\mathcal{F}$ for $(X\otimes_k E, \mathcal{O}_{X\otimes_k E}(r))$. Let $\pi:X\otimes_k E\rightarrow X$ be the projection, then $\pi_*\mathcal{F}$ is an Ulrich bundle for $(X,\mathcal{L}^{\otimes d})$. 
\end{proof}
\begin{thm}
	Let $X$ be an ivolution variety as above. Let $E$ be a minimal splitting field of $A$ and denote the base change $X\otimes_k E$ by $Q$. Then there is a very ample line bundles $\mathcal{L}$ on $X$, satisfying $\mathcal{L}\otimes_k E\simeq \mathcal{O}_{Q}(r)$ for suitable $r>0$, such that for any $d>0$ 
\begin{eqnarray*}
\mathrm{uc}((Q,\mathcal{O}_{Q}(rd)))\leq \mathrm{uc}((X,\mathcal{L}^{\otimes d}))\leq \mathrm{ind}(A)\cdot \mathrm{uc}((Q,\mathcal{O}_{Q}(rd))).
\end{eqnarray*}		
\end{thm}
\begin{rema}
\textnormal{If $(A,\sigma)$ is a central simple $\mathbb{R}$-algebra with involution of orthogonal type, $\mathrm{ind}(A)=2$. This is because the minimal separable splitting field is $\mathbb{C}$. We therefore have
\begin{eqnarray*}
	\mathrm{uc}((Q,\mathcal{O}_{Q}(rd)))\leq \mathrm{uc}((X,\mathcal{L}^{\otimes d}))\leq 2\cdot \mathrm{uc}((Q,\mathcal{O}_{Q}(rd))).
\end{eqnarray*}
}	
\end{rema}
\noindent
Let $X$ be one of the following involution varieties:
\begin{itemize}
	\item[(\textbf{i})] a twisted quadric associated to a central simple algebra $(A,\sigma)$ with involution of orthogonal type having trivial discriminant $\delta(A,\sigma)$, 
	\item[(\textbf{ii})] a twisted quadric associated to a central simple $\mathbb{R}$-algebra $(A,\sigma)$ with involution of orthogonal type.
\end{itemize}
Then one has $\mathrm{rdim}(X)+1\leq \mathrm{ind}(A)$ with equality for instance if $\mathrm{ind}(A)\leq 3$ (see \cite{NV2}, Theorem 6.16). This implies:
\begin{cor}
Let $X$ be one of the varieties (i) or (ii) and assume $\mathrm{rdim}(X)+1=\mathrm{ind}(A)$. Then there is a very ample line bundles $\mathcal{L}$ on $X$, satisfying $\mathcal{L}\otimes_k E\simeq \mathcal{O}_{Q}(r)$ for suitable $r>0$, such that for any $d>0$ 
\begin{eqnarray*}
	\mathrm{uc}((Q,\mathcal{O}_{Q}(rd)))\leq \mathrm{uc}((X,\mathcal{L}^{\otimes d}))\leq (\mathrm{rdim}(X)+1)\cdot \mathrm{uc}((Q,\mathcal{O}_{Q}(rd))).
\end{eqnarray*}		
\end{cor}
\begin{exam}
	\textnormal{For a smooth quadric $Q\subset \mathbb{P}^{n-1}$ embedded by $\mathcal{O}_Q(1)$, it is well known that $\mathrm{uc}((Q,\mathcal{O}_Q(1))=2^{\lfloor (n-3)/2\rfloor}$ (see \cite{BV}). From Proposition 3.3, we conclude that $\mathrm{uc}((Q,\mathcal{O}_Q(d))\leq (n-2)!2^{\lfloor (n-3)/2\rfloor}$. Let $X$ be a twisted quadric as in Corollary 5.17. Then $\mathrm{ind}(A)\geq 2$. Corollary 5.17 translates to
	$$
2^{\lfloor (\mathrm{dim}(X)-1)/2\rfloor}\leq \mathrm{uc}((X,\mathcal{L}^{\otimes d}))\leq \mathrm{ind}(A)\cdot \mathrm{dim}(X)\cdot 2^{\lfloor (\mathrm{dim}(X)-1)/2\rfloor}. 
$$
We want to mention that one can determine the minimal $r$ such that $\mathcal{L}\otimes_k E\simeq \mathcal{O}_Q(r)$, using the results in \cite{DT}. Moreover, for any scheme $X$, there is the notion of $\mathrm{ind}(X)$ which is defined as the greatest common divisor of the degrees of all finite field extensions $L$ with $X(L)\neq \emptyset$. So in general it may be possible to find a smaller upper bound using $\mathrm{ind}(X)$ instead of $\mathrm{ind}(A)$. For details in the case of involution varieties we refer to \cite{KRASH}.}
\end{exam}
\section{A relation between Ulrich complexity and $\mathrm{rdim}$}
In this section we make some elementary observations regarding the relation between categorical representability dimension and Ulrich complexity in the case of Brauer--Severi varieties. In fact, with the results of this section, we start to move towards an answer to questions 3) and 4).\\

\noindent
Let $X$ be a $n$-dimensional Brauer--Severi variety over a field $k$ corresponding to a central simple algebra $A$. Denote by $p$ the period of $X$ and let $\mathcal{E}$ be an Ulrich bundle for $(X,\mathcal{O}_X(pd))$, where $d\geq 1$ is a fixed integer. Then $\bar{\mathcal{E}}$ is an Ulrich bundle for $(\mathbb{P}^n,\mathcal{O}_{\mathbb{P}^n}(pd))$. Now for an Ulrich bundle on $\mathbb{P}^n$, the Euler-characteristic is given by 
$$
\chi(\bar{\mathcal{E}}(l))=\frac{\mathrm{rk}(\mathcal{E})}{n!}(l+pd)\cdot(l+2pd)\cdots (l+npd) \quad \forall l\in \mathbb{Z}.
$$
In particular, for $l=0$, we find
$$
\chi(\bar{\mathcal{E}})=\mathrm{rk}(\mathcal{E})\cdot (pd)^n
$$
and hence
$$
\frac{\chi(\bar{\mathcal{E}})}{d^n}=\mathrm{rk}(\mathcal{E})\cdot p^n.
$$
Throughout this section we use $\bar{\mathcal{E}}$ instead $\mathcal{E}$. But we could use $\mathcal{E}$ or any base change $\mathcal{E}\otimes_k E$. The aim is to relate the Ulrich complexity on $X$ to the rank of the corresponding vector bundle on $\mathbb{P}$. Since the period divides the index, we obtain
\begin{thm}
Let $X$ be as above and let $pd=\mathrm{ind}(X)\cdot t$. Denote by $\mathcal{E}$ an Ulrich bundle on $\phi_{pd}(X)$. Then
$$
\frac{\chi(\bar{\mathcal{E}})}{t^n}\geq \mathrm{uc}(\phi_{pd}(X))\cdot (\mathrm{rdim}(X)+1)^n.
$$
In particular, if $t=1$, then
$$
\chi(\bar{\mathcal{E}})\geq\mathrm{uc}(\phi_{pd}(X))\cdot (\mathrm{rdim}(X)+1)^n.
$$
\end{thm}
\begin{proof}
We have 	
$$
\chi(\bar{\mathcal{E}})=\mathrm{rk}(\mathcal{E})\cdot (pd)^n
$$
and therefore
$$
\chi(\bar{\mathcal{E}})=\mathrm{rk}(\mathcal{E})\cdot (\mathrm{ind}(X)\cdot t)^n.
$$
Since $\mathrm{rdim}(X)+1\leq \mathrm{ind}(X)$ (see \cite{NV3}, Proposition 4.1 and Theorem 1.4), we conclude 
$$
\frac{\chi(\bar{\mathcal{E}})}{t^n}\geq \mathrm{uc}(\phi_{pd}(X))\cdot (\mathrm{rdim}(X)+1)^n.
$$
\end{proof}
If period eqals index, it is conjectured in \cite{NV2} that $\mathrm{rdim}(X)+1=\mathrm{ind}(X)$. So it is reasonable to state
\begin{cor}
	Let $X$ be as above and assume $\mathrm{rdim}(X)+1=\mathrm{ind}(X)=p$. Denote by $\mathcal{E}$ an Ulrich bundle on $\phi_{pd}(X)$. Then
	$$
	\frac{\chi(\bar{\mathcal{E}})}{d^n}=\mathrm{rk}(\mathcal{E})\cdot (\mathrm{rdim}(X)+1)^n.
	$$
	In particular, if $d=1$, then
	$$
	\chi(\bar{\mathcal{E}})=\mathrm{rk}(\mathcal{E})\cdot (\mathrm{rdim}(X)+1)^n.
	$$
\end{cor}

\begin{thm}
		Let $X$ be as above and assume $\mathrm{rdim}(X)+1=p$. Denote by $\mathcal{E}$ an Ulrich bundle on $\phi_{pd}(X)$. Then 
		$$
		\mathrm{uc}(\phi_{pd}(X))=\mathrm{rdim}(X)+1
		$$
		if and only if
		$$
	\mathrm{uc}(\phi_{pd}(X))=\frac{1}{d}\sqrt[n]{ \frac{\chi(\bar{\mathcal{E}})}{\mathrm{rk}(\mathcal{E})} }.
	$$
	In particular, 
	$$
	\mathrm{uc}(\phi_{p}(X))=\mathrm{rdim}(X)+1
	$$
	if and only if
	$$
		\mathrm{uc}(\phi_{p}(X))=\sqrt[n]{ \frac{\chi(\bar{\mathcal{E}})}{\mathrm{rk}(\mathcal{E})} }.
	$$	
\end{thm}
\begin{proof}
Assume $\mathrm{uc}(\phi_{pd}(X))=\mathrm{rdim}(X)+1$. Then with 
$$
\frac{\chi(\bar{\mathcal{E}})}{d^n}= \mathrm{rk}(\mathcal{E})\cdot (\mathrm{rdim}(X)+1)^n
$$	
we obtain
	$$
\frac{\chi(\bar{\mathcal{E}})}{d^n}=\mathrm{rk}(\mathcal{E})\cdot \mathrm{uc}(\phi_{pd}(X))^{n}.
$$
Resolving to $\mathrm{uc}(\phi_{pd}(X))$ yields
$$
\mathrm{uc}(\phi_{pd}(X))=\frac{1}{d}\sqrt[n]{ \frac{\chi(\bar{\mathcal{E}})}{\mathrm{rk}(\mathcal{E})} }.
$$
On the other hand, if the latter equality holds, then 
$$
\frac{\chi(\bar{\mathcal{E}})}{d^n}=\mathrm{rk}(\mathcal{E})\cdot \mathrm{uc}(\phi_{pd}(X))^{n}.
$$
And since
$$
\frac{\chi(\bar{\mathcal{E}})}{d^n}= \mathrm{rk}(\mathcal{E})\cdot (\mathrm{rdim}(X)+1)^n,
$$
we get
$$
\mathrm{rk}(\mathcal{E})\cdot \mathrm{uc}(\phi_{pd}(X))^{n}=\mathrm{rk}(\mathcal{E})\cdot (\mathrm{rdim}(X)+1)^n
$$
implying $\mathrm{uc}(\phi_{pd}(X))=\mathrm{rdim}(X)+1$.
\end{proof}
\begin{cor}
Let $X$ be a non-split Brauer--Severi variety of dimension $p-1$, where $p$ denotes an odd prime. Assume $\mathrm{ind}(X)\leq 3$. Then
$$
\chi(\bar{\mathcal{E}})=\mathrm{rk}(\mathcal{E})\cdot (\mathrm{rdim}(X)+1)^{p-1}.
$$
Hence 
$$
\mathrm{uc}(\phi_{p}(X))=\mathrm{rdim}(X)+1
$$
if and only if
$$
\mathrm{uc}(\phi_{p}(X))=\sqrt[p-1]{ \frac{\chi(\bar{\mathcal{E}})}{\mathrm{rk}(\mathcal{E})} }.
$$	
	\end{cor}

\begin{exam}
	\textnormal{We consider the case of non-split Brauer--Severi curves $C$. In this case we know that $p=2=\mathrm{ind}(C)$ and $\mathrm{rdim}(C)+1=\mathrm{uc}(\phi_2(C))=2$. Now we want to use Corollary 6.4. Note that $\mathcal{V}_1$ is the unique rank two Ulrich bundle on $\phi_2(C)$ (see \cite{NV1}, Proposition 5.4). All Ulrich bundles on $C$ are given as $\mathcal{V}_1^{\oplus m}$. After base change, we have $\mathcal{V}_1\otimes_k \bar{k}\simeq \mathcal{O}_{\mathbb{P}^1}(1)^{\oplus 2}$. One can check that $\chi(\mathcal{O}_{\mathbb{P}^1}(1)^{\oplus 2})= 4$ and therefore $\chi(\mathcal{V}_1^{\oplus m})=4m$. Since $\mathrm{rk}(\mathcal{V}_1^{\oplus m})=2m$, we obtain
	$$
	\frac{\chi(\bar{\mathcal{E}})}{\mathrm{rk}(\mathcal{E})} =2
	$$	
	for any Ulrich bundle on $C$. This implies $\mathrm{uc}(\phi_{2}(C))=\mathrm{rdim}(C)+1$. }
\end{exam}
Let $X$ be a non-split Brauer--Sevei variety of dimension $p-1$, where $p$ denotes a prime and consider again $\chi(\bar{\mathcal{E}}(l))$. Instead of setting $l=0$, we can also set $l=tp$, where $p$ is the period. This is necessary because $\mathcal{O}_X(p)$ generates the Picard group. In this case, we have
$$
\chi(\bar{\mathcal{E}}(tp))=\mathrm{rk}(\mathcal{E})\cdot p^{p-1}\cdot \binom{p+t-1}{p-1}.
$$
Abbreviating the binomial coefficient by $b(p+t,p)$, one can show that 
$$
\mathrm{rdim}(X)+1=\mathrm{uc}(\phi_p(X))
$$
if and only if
$$
\mathrm{uc}(\phi_{p}(X))=\sqrt[p-1]{ \frac{\chi(\bar{\mathcal{E}})}{\mathrm{rk}(\mathcal{E})\cdot b(p+t,p)} }
$$	
\noindent
For $t=0$ we get back Corollary 6.4.
In the case of generalized Brauer--Severi varieties there are similar results as for Brauer--Severi varieties. So let $X=\mathrm{BS}(d,A)$ be a generalized Brauer--Severi variety embedded by  $\mathrm{BS}(m,A)\rightarrow \mathbb{P}(\wedge^{mn}(A))$ via the very ample line bundle $\mathcal{M}$. We can scale by an positive integer $e$ an embed $X$ via $\mathcal{M}^{\otimes e}$. Note that $n=\mathrm{deg}(A)$. The degree of $X$ is then given by $\mathrm{deg}(X,\mathcal{M}^{\otimes e})=e^{\mathrm{dim}(X)}\mathrm{deg}(X,\mathcal{M})$. It is well known that the Euler characteristic of an Ulrich sheaf $\mathcal{E}$ on $X$ is given as
$$
\chi(\mathcal{E}(t))=\mathrm{rk}(\mathcal{E})\cdot \mathrm{deg}(X)\cdot \binom{t+\mathrm{dim}(X)}{\mathrm{dim}(X)}.
$$
As mentioned before, for $\mathrm{BS}(m,A)$ there exists a finite Galois field extension $E$ such that $\mathrm{BS}(m,A)\otimes_k E\simeq \mathrm{Grass}_E(mn,n^2)\simeq \mathrm{Grass}_E(m,n)$. The Picard group $\mathrm{Pic}(\mathrm{Grass}_E(m,n))$ is isomorphic to $\mathbb{Z}$ and has ample generator $\mathcal{O}(1)\simeq \mathrm{det}(\mathcal{Q})$ with $\mathcal{Q}$ being the universal quotient bundle on $\mathrm{Grass}_E(m,n)$. Since $\mathrm{Pic}(\mathrm{BS}(m,A))\simeq\mathbb{Z}$, there is a positive generator $\mathcal{L}$ such that $\mathcal{L}\otimes_k E\simeq \mathcal{O}(r)$ for a suitable $r>0$. The smallest such $r$ is called the period of $X$. Now from \cite{NV}, Remark 7.3 we conclude that $r=\mathrm{per}(A^{\otimes m})$. Now let $\mathcal{E}$ be an Ulrich bundle for $(X,\mathcal{M}^{\otimes e})$. Then $\mathcal{E}_E:=\mathcal{E}\otimes_k E$ is Ulrich on $\mathrm{Grass}_E(m,n)$. Since $\mathcal{L}^{\otimes s}=\mathcal{M}$, the Euler characteristic of $\mathcal{E}_E$ becomes
$$
\chi(\mathcal{E}_E(t))=\frac{\mathrm{rk}(\mathcal{E})}{\mathrm{dim}(X)!}\cdot \mathrm{deg}(\mathrm{Grass}_E(m,n)))\cdot (s\cdot e\cdot \mathrm{per}(A^{\otimes m}))^{\mathrm{dim}(X)}\cdot \prod_{i=1}^{\mathrm{dim}(X)}(t+i).
$$
In particular, for $t=0$ we have
$$
\chi(\mathcal{E}_E)=\mathrm{rk}(\mathcal{E})\cdot \mathrm{deg}(\mathrm{Grass}_E(m,n)))\cdot (s\cdot e\cdot \mathrm{per}(A^{\otimes m}))^{\mathrm{dim}(X)}.
$$
Note that $\mathrm{per}(A^{\otimes m})=\mathrm{per}(A)/\mathrm{gcd}(\mathrm{per}(A),m)$. By setting $\mathrm{per}(A)=p$, we have 
$$
\chi(\mathcal{E}_E)=\mathrm{rk}(\mathcal{E})\cdot \mathrm{deg}(\mathrm{Grass}_E(m,n)))\cdot ( \frac{s\cdot e\cdot p}{\mathrm{gcd}(p,m)})^{\mathrm{dim}(X)}.
$$
And if the period is a prime, we get
$$
\chi(\mathcal{E}_E)=\mathrm{rk}(\mathcal{E})\cdot \mathrm{deg}(\mathrm{Grass}_E(m,n))\cdot ( s\cdot e\cdot p)^{\mathrm{dim}(X)}.
$$
Let us write $\mathbb{G}_{m,n}$ for $\mathrm{Grass}_E(m,n)$. 
\begin{thm}
	Let $X$ be the generalized Brauer--Severi variety from above over a field $k$ of characteristic zero. Assume $\mathrm{rdim}(X)+1=p$. Denote by $\mathcal{E}$ an Ulrich bundle on $(X,\mathcal{M}^{\otimes e})$. Then
	$$
	\chi(\mathcal{E}_E)\cdot (\frac{\mathrm{gcd}(p,m)}{s\cdot e})^{\mathrm{dim}(X)}\geq \mathrm{uc}((X,\mathcal{M}^{\otimes e}))\cdot \mathrm{deg}(\mathbb{G}_{m,n})\cdot (\mathrm{rdim}(X)+1)^{\mathrm{dim}(X)}.
	$$
\end{thm}
\begin{thm}
		Let $X$ be the generalized Brauer--Severi variety from above over a field $k$ of characteristic zero. Assume $\mathrm{rdim}(X)+1=\mathrm{ind}(A)=p$. Denote by $\mathcal{E}$ an Ulrich bundle on $(X,\mathcal{M}^{\otimes e})$. Then
		$$
		\mathrm{uc}((X,\mathcal{M}^{\otimes e}))=\mathrm{rdim}(X)+1
		$$
		if and only if
		$$
		\mathrm{uc}((X,\mathcal{M}^{\otimes e}))=(\frac{\mathrm{gcd}(p,m)}{s\cdot e})\cdot \sqrt[\mathrm{dim}(X)]{ \frac{ \chi(\mathcal{E}_E)}{\mathrm{rk}(\mathcal{E})\cdot \mathrm{deg}(\mathbb{G}_{m,n}) } }.
		$$
\end{thm}
The proofs of both theorems are completely analogous to the proofs of the corresponding statements for Brauer--Severi varieties. Similar results can also be formulated for involution varieties. In this case, the period needs to be determind, which is possible using results form \cite{DT}. 
\begin{rema}
	\textnormal{In the present paper we have seen that sometimes $\mathrm{uc}=\mathrm{rdim}(X)+b$, where $b\in\mathbb{Z}$. And it could also be that $\mathrm{uc}=c\cdot \mathrm{rdim}(X)$ for some positive integer $c$. For Brauer--Severi varieties one can formulate a statement which is analogue to Theorem 6.3 and considers the situation where $\mathrm{uc}=c\cdot \mathrm{rdim}(X)+b$. We omit the proof because it is nearly identical to the proof of Theorem 6.3. So, let $X$ be a Brauer--Severi variety of period $p$ and assume $\mathrm{rdim}(X)+1=p$. Let $\mathcal{E}$ be an Ulrich bundle on $\phi_pd(X)$. Then 
		$$
		\mathrm{uc}(\phi_{pd}(X))=c\cdot\mathrm{rdim}(X)+b
		$$
		if and only if
		$$
		\mathrm{uc}(\phi_{pd}(X))=(b-c)+\frac{c}{d}\sqrt[n]{ \frac{\chi(\bar{\mathcal{E}})}{\mathrm{rk}(\mathcal{E})}}.	
		$$ }
	\end{rema}
\section{Further examples}
One can ask whether $\mathrm{rdim}+1=\mathrm{uc}$ or $\mathrm{rdim}=\mathrm{uc}$ holds for other than Brauer--Severi varieties, twisted flags or including involution varieties. One other example is that of a del Pezzo surface $X$ over an algebraically closed field. In this case $\mathrm{uc}(X)=1$ and since a del Pezzo surface has a full exceptional collection, implying $\mathrm{rdim}(X)=0$, we find $\mathrm{uc}(X)=\mathrm{rdim}(X)+1$. But let us start more systematically. Without any claim to completeness, we give a list of examples:\\

\noindent
\textbf{curves}:
\noindent
Let $C$ be a smooth curve over a field $k$. In the case $g(C)=0$, the curve is a Brauer--Severi variety of period $p=1$ or $p=2$. We mentioned in the introduction that $\mathrm{rdim}(C)+1=\mathrm{uc}(\phi_p(C))$. Another interesting example appears when considering more general genus zero curves, namely so called twisted ribbons. These twisted ribbons are treated in Theorem 7.1 below. In the case where $g(C)\geq 1$, we have $\mathrm{rdim}(C)=1$, since there are no non-trivial semiorthogonal decompositions (see \cite{SOK}). It is known that the  Ulrich complexity of such a curve is one or two, depending on whether the curve admits rational points or not. So we find that either $\mathrm{rdim}(C)+1=\mathrm{uc}(C)$ or $\mathrm{rdim}(C)=\mathrm{uc}(C)$ in case of any curve $C$. We also see that the nature of the relation between $\mathrm{rdim}$ and $\mathrm{uc}$ also depends on the arithmetic of the curve respectively the ground field $k$.\\

\noindent
Let us now consider the case of surfaces. The case of del Pezzo surfaces over an algebraically closed field was mentioned before. For del Pezzo surfaces over arbitrary base fields, the situation, to our best knowledge, is not yet completely understood. We give a list of known cases below. Recall, that by a del Pezzo surface over $k$ we mean a smooth, projective and geometrically integral surface with ample anticanonical class $\omega_S$. The degree of $S$ is the self-intersection number of $\omega_S$.\\
\textbf{del Pezzo surface of degree 9}:\\
a del Pezzo surface $S$ of degree 9 is a Brauer--Severi surface corresponding to a central simple algebra $A$ of degree 3. In this case $\mathrm{ind}(A)\leq 3$ and hence $\mathrm{rdim}(S)+1=\mathrm{ind}(A)$. The Ulrich complexity is still to be established, but Example 5.3 gives lower and upper bounds.\\
\textbf{del Pezzo surfaces of degree 8}:\\
del Pezzo surfaces of degree 8 are involution surfaces (see \cite{BERAU}). These surfaces are treated in Theorems 7.3, 7.4 and 7.5, where the relation between $\mathrm{rdim}$ and $\mathrm{uc}$ is completely determined.\\
\textbf{del Pezzo surfaces of degree 7}:\\
If $S$ is a del Pezzo surface $S$ of degree 7, one has $\mathrm{rdim}(S)+1=1$ (see \cite{BERAU}, p.36). Such a del Pezzo surface is the blow up of $\mathbb{P}^2$ at a closed point of degree 2. Notice that $\mathrm{uc}((\mathbb{P}^2,\mathcal{O}_{\mathbb{P}^2}(d)))=2$ if $d$ is even, according to \cite{CMR}. One can then use \cite{ASE}, Theorem 2 to obtain $\mathrm{rdim}(S)+2=\mathrm{uc}((\mathbb{P}^2,\mathcal{O}_{\mathbb{P}^2}(d)))$ in this case. Note that $d$ even is used to apply the results of \cite{ASE}.\\
\noindent
\textbf{del Pezzo surfaces of degree $\leq 6$}:\\
We don't know.\\
\noindent
\textbf{minimal ruled surfaces}:
Let $S$ be a minimal ruled surfaces over $\bar{k}$. In this case $D^b(S)=\langle D^b(C),D^b(C)\rangle$, where $C$ is the base curve. If $g(C)=0$, $\mathrm{rdim}(S)=0$. If $g(C)\geq 1$, $\mathrm{rdim}(S)=1$, since curves of genus at least one do not admit non-trivial semiorthogonal decompositions. In the case $g(C)=0$, we have $\mathrm{uc}((S,\mathcal{L}))=1$ for certain ample line bundles $\mathcal{L}$. This follows from the results in \cite{HOCH}. In the case $g(C)\geq 1$, we have that $\mathrm{uc}((S,\mathcal{L}))$ is 1 or 2, depending on the polarization (see \cite{ACMRV}). This shows that either $\mathrm{rdim}(S)+1=\mathrm{uc}(S)$ or $\mathrm{rdim}(S)=\mathrm{uc}(S)$ in the case of certain minimal ruled surfaces, too.\\
\noindent
\textbf{abelian surface}:\\
It is known that $\mathrm{rdim}(S)=2$ if $S$ is an abelian surface. This is because $D^b(S)$ has non non-trivial semiorthogonal decompositions (see \cite{KUZ}). Now from \cite{BEV} we conclude $\mathrm{uc}((S,\mathcal{L}))\leq 2$. If $S$ is of rank one, $\mathrm{uc}((S,\mathcal{L}))=2$. This shows that either $\mathrm{rdim}(S)=\mathrm{uc}(S)$ or $\mathrm{rdim}(S)-1=\mathrm{uc}(S)$ in the case of abelian surfaces.\\
\noindent
\textbf{K3 surfaces}:\\
Let $S$ be a $K3$ surface. It is known that $D^b(S)$ has no non-tivial semiorthogonal decomposition (see \cite{KUZ}) and hence $\mathrm{rdim}(S)=2$. In \cite{DF0} it is shown that $\mathrm{uc}(S)\leq 2$ for any polarization on $S$. This shows that either $\mathrm{rdim}(S)=\mathrm{uc}(S)$ or $\mathrm{rdim}(S)-1=\mathrm{uc}(S)$ in the case of $K3$ surfaces.\\
\noindent
\textbf{surfaces with $q_g=q=0$}:\\
In this case it can happen that the derived category of the surface $S$ admits a so called \emph{phantom} category, being the complementary component of a exceptional collection which is not full. These complementary components have finite (or even zero) Grothendieck group and trivial Hochschild homology and by that reason they are called phnatoms. Phantoms are proved to exist for instance for classical Godeaux surfaces, Beauville surfaces or Burniat surfaces (see \cite{KUZ} and references therein). So for these surfaces $\mathrm{rdim}(S)\leq 2$. For certain ample line bundles $\mathcal{L}$ it can be proved that $\mathrm{uc}((S,\mathcal{L}))\leq 2$ (see \cite{BVIL} and \cite{CASN}): This shows that either $\mathrm{rdim}(S)=\mathrm{uc}((S,\mathcal{L}))$ or $\mathrm{rdim}(S)-1=\mathrm{uc}((S,\mathcal{L}))$ or $\mathrm{rdim}(S)+1=\mathrm{uc}((S,\mathcal{L}))$.\\
\noindent
\textbf{Fano 3-folds}:\\
In dimension three, we only want to mention one example stated previously in the present paper, namely that of a Fano 3-fold $X$ of index 2. If $X$ is defined over arbitrary $k$, it could be a Brauer--Severi 3-fold and the Corollary from the intruduction shows $\mathrm{uc}((X,\mathcal{O}_X(2)))=2$. Since the degree of the corresponding central simple algebra $A$ is 4, $\mathrm{ind}(A)\leq 4$. If $k=\mathbb{R}$, then $\mathrm{ind}(A)=2$ and we know that $\mathrm{rdim}(X)+1=2$. This shows $\mathrm{rdim}(X)+1=\mathrm{uc}((X,\mathcal{O}_X(2)))$. Over $\bar{k}$, \cite{BV}, Proposition 6.1 gives $\mathrm{uc}((X,\mathcal{O}_X(2)))\leq 2$. Note that $\mathrm{rdim}(X)\leq 3$. 

\noindent
\textbf{twisted ribbons}:\\
Another example is given when considering regular genus zero curves. Brauer--Severi curves are not the only examples of such curves. The Ulrich complexity can also be determined for so called ribbons of genus zero. Following Bayer and Eisenbud \cite{BYE}, we say that a \emph{ribbon} on a curve $C_0$ is a pair $(C,\iota)$, where $\iota:C_0\rightarrow C$ is a closed embedding having a square-zero sheaf of ideals $\mathcal{I}$ that is invertible as $\mathcal{O}_{C_0}$-module. The ribbons splits if $\iota : C_0\rightarrow C$ has a retraction $\rho: C\rightarrow C_0$  and write $C=C_0\oplus \mathcal{L}$ for some invertible $\mathcal{O}_{C_0}$-module $\mathcal{L}$. Now following \cite{FS} we call a curve $C$ with invariants $h^0(\mathcal{O}_C)=1$ and $h^1(\mathcal{O}_C)=0$ a genus zero curve. In \emph{loc. cit.} it is shown in Proposition 1.2 that if $C$ is regular it must be a twisted form of $\mathbb{P}^1$ or a twisted form of the split ribbon $\mathbb{P}^1\oplus \mathcal{O}_{\mathbb{P}^1}(-1)$. Now let $D$ is a regular genus zero curve without rational point whose Picard group is generated by the dualizing sheaf $\omega_D$. The vector space $\mathrm{Ext}^1(\mathcal{O}_D,\omega_D)$ is one dimensional and the non-split exact sequence 
$$
0\rightarrow \omega_D\rightarrow \mathcal{F}_D\rightarrow \mathcal{O}_D\rightarrow 0
$$
defines a vector bundle $\mathcal{F}_D$ of rank two. Up to isomorphism it does not depend on the choise of the extension and is canonically attached to $D$. In \cite{FS}, Theorem 10.2 all indecomposable locally free sheaves on $D$ were classified. In particular, all indecomposable vector bundles are of the form $\omega_D^{\otimes a}$ and $\mathcal{F}_D\otimes \omega_D^{\otimes b}$, with $a,b\in \mathbb{Z}$.

\begin{thm}
	Let $D$ be a regular zero one curve without rational point whose Picard group is generated by the dualizing sheaf, that is a Brauer--Severi curve or a twisted ribbon. Then $\mathrm{uc}((D,\omega_D^{\otimes -d}))=2$. 
\end{thm}
\begin{proof}
The result for Brauer--Severi curves can be found in \cite{NV1}. Let us prove the statement for twisted ribbons. Let $\mathcal{E}=\mathcal{F}_D\otimes \omega_D^d$. Then $h^0(\mathcal{E}\otimes \omega_D^{-d})=h^0(\mathcal{F}_D)=0$ and $h^1(\mathcal{E}\otimes \omega_D^{-d})=h^1(\mathcal{F}_D)=0$ (see \cite{FS}, page 30). This shows $\mathrm{uc}((D,\omega_D^{\otimes -d}))\leq 2$. Now we show that there is no Ulrich line bundle on $D$. Note that all line bundles are of the form $\omega_D^{\otimes a}$. Let us consider the vector spaces $h^0(\omega_D^{\otimes (a-d)})$ and $h^1(\omega_D^{\otimes (a-d)})$. Now from \cite{FS}, Proposition 10.3 we obtain  $h^0(\omega_D^{\otimes (a-d)})=0$ if and only if $a-d>0$. Considering $h^1(\omega_D^{\otimes (a-d)})=h^0(\omega_D^{\otimes (1-(a-d))})$, we see that $1-(a-d)\leq 0$ and therefore  $h^1(\omega_D^{\otimes (a-d)})\neq 0$ if $a-d>0$. This shows that there cannot be an Ulrich line bundle. Hence $\mathrm{uc}((D,\omega_D^{\otimes -d}))=2$.
\end{proof}	
\begin{rema}
	\textnormal{The categorical representability dimension of twisted ribbons of genus zero that are not Brauer--Severi curves are still to be established. We believe that $\mathrm{rdim}(D)=1$. A brief argument goes as follows: a semiorthogonal decopmposition of $D^b(D)$ should be given by $D^b(D)=\langle \mathcal{O}_D,\mathcal{F}_D\rangle$. Now $\langle \mathcal{F}_D\rangle \simeq D^b(\mathrm{End}(\mathcal{F}_D))$. But $ D^b(\mathrm{End}(\mathcal{F}_D))$ cannot be an admissible subcategory of $D^b(k)$. Hence $\mathrm{rdim}(D)=1$. Consequently, we conjecture:
	$$
\mathrm{rdim}(D)+1=\mathrm{uc}((D,\omega_D^{\otimes -d})).
$$ }
\end{rema}
\noindent
\textbf{involution surfaces or del Pezzo surfaces of degree 8}:\\
As mentioned above, in the case of surfaces one can also study involution surfaces over $k$. Such involution surfaces are del Pezzo surfaces of degree 8 and are  twisted forms of $\mathbb{P}^1\times \mathbb{P}^1$. So roughly, it is the product of two Brauer--Severi curves $C_1\times C_2$. 
Note that the product of two Brauer--Severi varieties $X_1\times X_2$ is isomorphic to a projective bundle $\mathbb{P}(\mathcal{E})\rightarrow X_1$. So in our situation, there are three cases: $C\times C$ or $C\times \mathbb{P}^1$ or $C_1\times C_2$, where $C,C_1, C_2$ are non-split Brauer--Severi curves. Let us start with $S=C\times C$ and let us consider $S$ as a projective bundle $\pi:\mathbb{P}(\mathcal{E})\rightarrow C$. With this notation, we obtain:
\begin{thm}
	Let $X=C\times C$ be the product of a non-split Brauer--Severi curve with itself let $\mathcal{L}$ be a very ample line bundle of the form $\pi^{*}\mathcal{A}+H$, where $H$ is the relative hyperplane section. Then
	$$
	\mathrm{uc}((C\times C,\mathcal{L}))=2=\mathrm{rdim}(C\times C)+1.
	$$
\end{thm}
\begin{proof}
	As mentioned above, the product $C\times C$ is isomorphic to some $\mathbb{P}(\mathcal{E})$ over $C$. Now we can use the results in \cite{HOCH}, where Ulrich bundles on projective bundles were studied. In particular, we use \cite{HOCH}, Theorem 4.3. Let us take the indecomposable vector bundle $\mathcal{V}_{-1}$ on $C$. The vector bundle $\mathcal{V}_{-1}$ is obtained from the Euler sequence on $C$ and satisfies $\mathcal{V}_{-1}\otimes_k \bar{k}\simeq \mathcal{O}_{\mathbb{P}^1}(-1)^{\oplus 2}$. Now $H^{\bullet}(X,\mathcal{V}_{-1})=0$ and therefore, by \cite{HOCH}, Theorem 4.3, the vector bundle $(\pi^{*}\mathcal{V}_{-1})\otimes \mathcal{L}$ is Ulrich. Now we show that there is no Ulrich line bundle. Any line bundle on $\mathbb{P}(\mathcal{E})$ must be of the form $\pi^{*}\mathcal{M} +nH$, where $\mathcal{M}$ is a line bundle on $C$ and $H$ is the relative hyperplane section. After base change, the projective bundle over $C$ becomes the Hirzebruch surface $\mathbb{F}_0$. This follows from the following fact: It is known that $\mathcal{E}\otimes_k \bar{k}\simeq \mathcal{O}_{\mathbb{P}^1}(a)^{\oplus 2}$. And from this we get that $\mathbb{P}(\mathcal{E})\rightarrow C$ becomes isomorphic to $\mathbb{F}_0\rightarrow \mathbb{P}^1$ after base change.  
	Now from \cite{HOCH}, Remark 3.3 we conclude that any Ulrich line bundle on $\mathbb{F}_0$ is of the form $\pi^{*}\mathcal{F}\otimes \mathcal{L}$ for some line bundle $\mathcal{F}$ on $\mathbb{P}^1$. So starting with an Ulrich line bundle $\mathcal{D}$ on $C\times C$ we obtain an Ulrich line bundle $\mathcal{D}\otimes_k \bar{k}$ on $\mathbb{F}_0$. This line bundle is of the form $\pi^{*}\mathcal{F}\otimes \mathcal{L}$ for some line bundle $\mathcal{F}$ on $\mathbb{P}^1$. From \cite{HOCH}, Theorem 4.3 it follows that $H^{\bullet}(\mathbb{P}^1,\mathcal{F})=0$. This is possible only if $\mathcal{F}\simeq \mathcal{O}_{\mathbb{P}^1}(-1)$. But $\pi^{*}\mathcal{F}\otimes \mathcal{L}$ does not descent to some line bundle on $C\times C$. This shows that the Ulrich complexity is two. 
	To determine $\mathrm{rdim}$, we use results in \cite{NV2} and refer to the references therein for details. Let $\mathcal{V}_1$ be the dual of the indecomposable vector bundle $\mathcal{V}_{-1}$ on $C$ comming from the Euler sequence as mentioned above. 
 It can be shown that there is a semiorthogonal decomposition
	$$
	D^b(C\times C)=\langle \mathcal{O}_{C}\boxtimes \mathcal{O}_{C},\mathcal{V}_1\boxtimes \mathcal{O}_{C},\mathcal{O}_{C}\boxtimes \mathcal{V}_1,\mathcal{V}_1\boxtimes \mathcal{V}_1 \rangle, 
	$$ 
	where $\mathcal{F}\boxtimes \mathcal{G}$ denotes $p_1^{*}\mathcal{F}\otimes p_2^{*}\mathcal{G}$ for the projections $p_i:C_1\times C_2\rightarrow C_i$. Let $A$ be the quaternion algebra corresponding to $C$. Now from \cite{NV2}, Proposition 4.5, we conclude that $\mathrm{End}(\mathcal{O}_{C}\boxtimes \mathcal{O}_{C})=k, \mathrm{End}(\mathcal{V}_1\boxtimes \mathcal{O}_{C})=A, \mathrm{End}(\mathcal{O}_{C}\boxtimes \mathcal{V}_1)=A$ and $\mathrm{End}(\mathcal{V}_1\boxtimes \mathcal{V}_1)=M_2(k)$. And since all the admissible components $\mathcal{D}_i$ in the semiorthogonal decomposition are equivalent to the bounded derived category of the Endomorphism algebra of the corresponding sheaves, we conclude $\mathcal{D}_1\simeq D^b(k)$, $\mathcal{D}_2\simeq D^b(A)$, $\mathcal{D}_3\simeq D^b(A)$ and $\mathcal{D}_4\simeq D^b(k)$. Since a semoorthogonal decomposition of $D^b(C)$ is given as
	$$
	D^b(C)=\langle D^b(k), D^b(A)\rangle,
	$$
	it follows that $\mathcal{D}_i$ is equivalent to an admissible subcategory of $D^b(C)$. This shows $\mathrm{rdim}(C\times C)\leq 1$. Since $C$ is non-split, $C\times C$ does not have a rational point. Hence $\mathrm{rdim}(C\times C)=1$, according to \cite{NV2}, Theorem 6.3.
\end{proof}
\begin{thm}
	Let $X=C_1\times C_2$ be the product of two distinct non-split Brauer--Severi curves and $\mathcal{L}$ a very ample line bundle of the form $\pi^{*}\mathcal{A}+H$, where $H$ is the relative hyperplane section. Then
	$$
	\mathrm{uc}((X,\mathcal{L}))=\mathrm{rdim}(X).
	$$
\end{thm}
\begin{proof}
If $X=C_1\times C_2$ is the product of two distinct non-aplit Brauer--Severi curves, it is still possible to show $\mathrm{uc}((X,\mathcal{L}))=2$. The arguments are pretty much the same as in the proof of Theorem 7.3. But when it comes to $\mathrm{rdim}(X)$, one can show that $\mathrm{rdim}(X)=2$. In fact one can show that there is a semiorthogonal decomposition  
$$
D^b(C_1\times C_2)=\langle \mathcal{O}_{C_1}\boxtimes \mathcal{O}_{C_2},\mathcal{V}_1\boxtimes \mathcal{O}_{C_2},\mathcal{O}_{C_1}\boxtimes \mathcal{W}_1,\mathcal{V}_1\boxtimes \mathcal{W}_1 \rangle, 
$$ 
where $\mathcal{V}_1$ is the dual of $\mathcal{V}_{-1}$ and $\mathcal{W}_1$ the dual of $\mathcal{W}_{-1}$. Note that $\mathcal{V}_{-1}$ is the vector bundle in the middle of the Euler sequence for $C_1$, and $\mathcal{W}_{-1}$ respectively. Again from \cite{NV2}, Proposition 4.5, we conclude that $\mathrm{End}(\mathcal{O}_{C_1}\boxtimes \mathcal{O}_{C_2})=k, \mathrm{End}(\mathcal{V}_1\boxtimes \mathcal{O}_{C_2})=A, \mathrm{End}(\mathcal{O}_{C_1}\boxtimes \mathcal{W}_1)=B$ and $\mathrm{End}(\mathcal{V}_1\boxtimes \mathcal{W}_1)=A\otimes B$, where $A$ and $B$ are the quaternion algebras corresponding to $C_1$ and $C_2$. Then \cite{PB}, Appendix B Corollary B.4. implies $\mathrm{rdim}(X)=2$.	
\end{proof}
\begin{thm}
	Let $X=C\times \mathbb{P}^1$ be the product of a non-split Brauer--Severi curve and a projective line and $\mathcal{L}$ a very ample line bundle of the form $\pi^{*}\mathcal{A}+H$, where $H$ is the relative hyperplane section. Then
	$$
	\mathrm{uc}((X,\mathcal{L}))=\mathrm{rdim}(X)+1.
	$$
\end{thm}
\begin{proof}
	As in the proof of Theorem 7.3 one can show $\mathrm{uc}((X,\mathcal{L}))=2$. The semiorthogonal decomposition is given as
	$$
D^b(C\times \mathbb{P}^1)=\langle \mathcal{O}_{C}\boxtimes \mathcal{O}_{\mathbb{P}^1},\mathcal{V}_1\boxtimes \mathcal{O}_{\mathbb{P}^1},\mathcal{O}_{C}\boxtimes \mathcal{O}_{\mathbb{P}^1}(1),\mathcal{V}_1\boxtimes \mathcal{O}_{\mathbb{P}^1}(1) \rangle, 	
	$$
where $\mathcal{V}_1$ is again the dual of $\mathcal{V}_{-1}$. Again one can show that the existence of the above semiorthogonal decomposition implies $\mathrm{rdim}(X)=1$.		
	\end{proof}
\section{Appendix}
Finally, let me give an idea of how a relation between $\mathrm{uc}$ and $\mathrm{rdim}$ can be established in a more gereral setting. At this point I want to mention that the following lines rather contain ideas and some intuition than rigorous or precise mathematical statements. But it could be the starting point of further work in this direction.\\

\noindent
Let $X$ be a smooth projective variety over $k$ and assume there is a full weak exceptional collection $\mathcal{E}_0,...,\mathcal{E}_n$, where $\mathrm{End}(\mathcal{E}_j)$ are central simple. Since $\langle \mathcal{E}_j\rangle\simeq D^b(\mathrm{End}(\mathcal{E}_j))$, it is reasonable to expect that $\mathrm{rdim}(X)$ depends on invariants of the central simple algebras $\mathrm{End}(\mathcal{E}_j)$. Notice that we are particularly interessted in embeddings of the form 
$$
D^b(\mathrm{End}(\mathcal{E}_i))\hookrightarrow D^b(Y_i)
$$
where $Y_i$ are smooth projective varieties with $\mathrm{dim}(Y_i)\leq m$. Now we want to find the smallest such $m$. Let us denote period and index of $\mathrm{End}(\mathcal{E}_j)$ by $p_j$ and $i_j$. So we expect $\mathrm{rdim}(X)=f(p_1,...,p_n,i_1,...,i_n)$ for a suitable function $f$. On the other hand, there is some Beilinson-type spectral sequence:
\begin{thm}
	Let $X$ be a smooth projective variety over a field $k$ and $\mathcal{E}_0,...,\mathcal{E}_n$ a full weak exceptional collection of coherent sheaves. Then for any coherent sheaf $\mathcal{F}$ there is a spectral sequence
	\begin{eqnarray*}
		E^{p,q}_1=\mathrm{Ext}^q(R_{\mathcal{E}_n\cdots\mathcal{E}_{p+n+1}}\mathcal{E}_{p+n},\mathcal{F})\otimes \mathcal{E}_{p+n}\Rightarrow E^{p+q}=\begin{cases}\mathcal{F}& \text{if} \;\;\;\;\; p+q=0\\
			0& \text{otherwise}
		\end{cases}
	\end{eqnarray*}
	The grading is bounded by $0\leq q\leq n$ and $-n\leq p\leq 0$.
\end{thm}
Now let $\mathcal{F}$ be a certain twist of an Ulrich bundle on $X$ and assume there is a vanishing theorem of the form
$$
\mathrm{Ext}^q(R_{\mathcal{E}_n\cdots\mathcal{E}_{p+n+1}}\mathcal{E}_{p+n},\mathcal{F})=0  \quad \textnormal{for} \quad q\neq 1 \quad \textnormal{and} \quad  h^1(\mathcal{F})=0.
$$
Then, by the properties of the above spectral sequence, we would get an exact sequence of the form
	\begin{eqnarray*}
	0\longrightarrow \mathcal{E}_{0}^{\oplus b_0}\longrightarrow \mathcal{E}_{1}^{\oplus b_{1}}\longrightarrow \cdots \longrightarrow \mathcal{E}_{n-1}^{\oplus b_{n-1}}\longrightarrow \mathcal{F}\longrightarrow 0.
\end{eqnarray*}
Now the rank of the vector bundles $\mathcal{E}_j$ and the dimension of the extension groups 
 $\mathrm{Ext}^1(R_{\mathcal{E}_n\cdots\mathcal{E}_{p+n+1}}\mathcal{E}_{p+n},\mathcal{F})$ depend on the period and/or the index of $\mathrm{End}(\mathcal{E}_j)$. This implies that $\mathrm{rk}(\mathcal{F})=g(p_1,...,p_n,i_1,...,i_n)$ for a suitable function $g$. Now one could study the relation between $f$ and $g$. If, for instance, one of the functions could be resolved to $p_j$ or $i_j$, then there is a direct relation between the Ulrich complexity and $\mathrm{rdim}(X)$. So if $p_1=h(\mathrm{rdim}(X),p_2,...p_n,i_0,...,i_n)$ for a suitable $h$, then by pluging $p_1$ into $g$ gives 
 $$
 \mathrm{uc}(X)=g(h(\mathrm{rdim}(X),p_2,...,p_n,i_0,...,i_n),p_2,...,p_n,i_0,...,i_n).
 $$ 
 For certain Brauer--Severi varieties $X$ with period $p$ equals index $i$, we had $\mathrm{rdim}(X)+1=\mathrm{uc}(\phi_p(X))$. We see that we could set $\mathrm{rdim}(X)=f(p,i)=p-1$ and $\mathrm{uc}(\phi_p(X))=g(p,i)=p$. In this case we do not need to resolve to $p$. Instead, we directly get $\mathrm{rdim}(X)+1=\mathrm{uc}(\phi_p(X))$. But it is not clear to us whether $f$ and $g$ are obtained from the above construction. 
 

\vspace{0.5cm}
\noindent
{\tiny HOCHSCHULE FRESENIUS UNIVERSITY OF APPLIED SCIENCES 40476 D\"USSELDORF, GERMANY.}\\
E-mail adress: sasa.novakovic@hs-fresenius.de\\
\noindent
{\tiny MATHEMATISCHES INSTITUT, HEINRICH--HEINE--UNIVERSIT\"AT 40225 D\"USSELDORF, GERMANY.}\\
E-mail adress: novakovic@math.uni-duesseldorf.de

\end{document}